\documentclass[12pt]{amsart}
\oddsidemargin 0cm \evensidemargin -0cm \textwidth 16cm\usepackage{amsmath,amsthm,amssymb,latexsym,epic,bbm,comment,yfonts,mathrsfs}
\usepackage{graphicx,enumerate,stmaryrd,rotating,xcolor,color,ytableau,soul}
	\input xy
\allowdisplaybreaks

\usepackage{amsxtra}
\usepackage{amsmath}
\usepackage{amscd}
\usepackage{amssymb}
\usepackage{amsfonts}
\usepackage{mathrsfs}
\usepackage{amsthm}
\usepackage{color}
\usepackage{upgreek}
\usepackage{verbatim}  %%for \begin{comment}
\usepackage{enumerate}
\usepackage{latexsym}
\usepackage{graphicx}
\usepackage{tikz}
\usepackage{todonotes}
\usepackage{setspace}
\usepackage{hyperref}
\usepackage{stmaryrd}
\usepackage{nicefrac}
\usepackage{euscript}
\usetikzlibrary{decorations.pathmorphing}
\usepackage{tikz}
\usepackage{tikz-cd}

 \definecolor{darkgreen}{HTML}{336633}
 \definecolor{darkred}{HTML}{993333}
\makeatletter

\newcommand{\arxiv}[1]{\href{http://arxiv.org/abs/#1}{\tt
    arXiv:\nolinkurl{#1}}}

\newtheorem{alphatheorem}{Theorem}

\theoremstyle{plain}
\newtheorem{thm}{Theorem}[section]
\newtheorem*{thm*}{Theorem}

\newtheorem{lem}[thm]{Lemma}
\newtheorem{prop}[thm]{Proposition}

\newtheorem{df-prop}[thm]{Definition-Proposition}

\theoremstyle{remark}
\newtheorem{rem}[thm]{Remark}
\newtheorem{ex}[thm]{Example}

\numberwithin{equation}{section}

\def\bbQ{\mathbb{Q}}

\def\bbT{\mathbb{T}}

\def\calF{\mathcal{F}}

\def\calO{\mathcal{O}}
\def\calP{\mathcal{P}}

\def\calR{\mathcal{R}}

\newcommand{\cZ}{\mathcal{Z}}

\def\bfs{\mathbf{s}}
\def\bft{\mathbf{t}}

\def\uep{\underline{\ep}}

\def\bfep{\boldsymbol{\epsilon}}
\def\bfla{\boldsymbol{\lambda}}
\def\bflaep{\boldsymbol{\epsilon\lambda}}
\def\bfmu{\boldsymbol{\mu}}
\def\bfmue{\boldsymbol{\epsilon\mu}}

\def\bfA{\mathbf{A}}

\def\U{\mathbf{U}}

\def\mod{\operatorname{-mod}\nolimits}

\def\wt{\operatorname{wt}\nolimits}

\def\ep{\epsilon}
\def\gl{\mathfrak{gl}}
\def\sln{\mathfrak{sl}}

\def\la{\lambda}

\def\ula{{\underline{\lambda}}}

\def\ov{\overline}

\def\Col{\text{Col}}
\def\Row{\text{Row}}
\def\Std{\text{Std}}

\newcommand{\brown}[1]{\textcolor{brown}{ #1}}
\newcommand{\green}[1]{\textcolor{green}{ #1}}

\newcommand{\orange}[1]{\textcolor{orange}{ #1}}

\newcommand{\mc}{\mathcal}
\newcommand{\mf}{\mathfrak}

\newcommand{\C}{\mathbb C}
\newcommand{\Q}{\mathbb Q}

\newcommand{\oa}{{\bar 0}}
\newcommand{\ob}{{\bar 1}}
\newcommand{\ad}{\mathrm{ad}}

\newcommand{\fg}{\mathfrak{g}}

\newcommand{\h}{\mathfrak{h}}
\newcommand{\N}{\mathbb{Z}_{\geq 0}}

\newcommand{\g}{\mathfrak{g}}

\newcommand{\fu}{\mathfrak{u}}

\newcommand{\Z}{{\mathbb Z}}

\newcommand{\glnm}{\mf{gl}_{n|m}}
\newcommand{\Ofull}{\mathcal{O}(\bflaep)}
\newcommand{\Pcat}{\mathcal{P}(\bflaep)}
\newcommand{\VV}{\mathbb V}

\newcommand{\Wg}{U(\g,\bflaep)}
\newcommand{\Wgl}{U(\glnm,\bflaep)}
\newcommand{\Wl}{U(\mf l,\bflaep)}

\begin{document}

\title[Simple character formulas for finite $W$-superalgebras]{Simple character formulas for finite $W$-superalgebras of type $A$}

\author[Shun-Jen Cheng]{Shun-Jen Cheng}
\address{Institute of Mathematics, Academia Sinica, Taipei, Taiwan 10617} \email{chengsj@math.sinica.edu.tw}
\author[Weiqiang Wang]{Weiqiang Wang}
\address{Department of Mathematics, University of Virginia, Charlottesville, VA 22903, USA}\email{ww9c@virginia.edu}


\begin{abstract}
We establish a canonical basis character formula for the irreducible modules in arbitrary parabolic BGG-type categories, including the category of finite-dimensional modules, for finite $W$-superalgebras of type $A$. These categories categorify the tensor product modules of irreducible polynomial representations and their duals over a quantum group of type $A$. Moreover, the standard modules and irreducible modules in these categories categorify the standard basis and Lusztig's dual canonical basis in the tensor product modules. Our formula provides a uniform generalization of several character formulas in BGG categories for Lie superalgebras and for $W$-algebras of type $A$.
\end{abstract}

\maketitle

\setcounter{tocdepth}{1}
\tableofcontents

\thispagestyle{empty}

%\vspace{1em}

%\begin{quote}
%\begin{center}
%{\em Dedicated to George Lusztig with admiration and appreciation.}
%\end{center}
%\end{quote}

\vspace{1em}

%%%%%%%%%%%
%%%%%%%%%%%
\section{Introduction}

\subsection{The goal}

Finite $W$-algebras $U(\g,e)$ associated with a nilpotent element $e$ of  a semisimple Lie algebra $\g$ can be viewed as a generalization of the enveloping algebras $U(\g)=U(\g,0)$, and they quantize Slodowy slices; see \cite{Los10, Wan11} for surveys. A BGG-type category $\calO$ of $U(\g,e)$-modules including Verma modules was introduced in \cite{BGK08}, with parabolic categories including parabolic Verma $U(\g,e)$-modules subsequently formulated in \cite{Los12} and \cite{BG13}. The counterparts of Cartan and Levi of $U(\g,e)$ are again finite $W$-algebras, and a novelty here is that they are subquotients instead of subalgebras of $U(\g,e)$. It is straightforward to extend these definitions to the general linear Lie superalgebras $\g =\glnm$, which will be the main focus of this paper.

A milestone in representation theory is the Kazhdan-Lusztig theory for the BGG category $\calO$ of semisimple Lie algebras \cite{KL79, BB81, BK81}, where the canonical basis of Hecke algebras plays a key role. Brundan \cite{Bru03} formulated a super Kazhdan-Lusztig conjecture for $\glnm$ (now a theorem of \cite{CLW15}; see also \cite{BLW17}), where Lusztig's canonical basis \cite{Lus92, Lus10} on tensor product modules $\VV^{\otimes n}\otimes (\VV^{-})^{\otimes m}$ of the quantum group $\U=\U(\gl_\infty)$ is used in place of canonical basis of Hecke algebras; here $\VV=\VV^+$ denotes the natural representation of $\U$ and $\VV^-$ its restricted dual. For the usual type $A$, this amounts to a reformulation of parabolic Kazhdan-Lusztig basis \cite{KL79, Deo87} in terms of canonical basis on the tensor module $\VV^{\otimes n}$ \cite{FKK98} via the Jimbo-Schur duality \cite{Jim86}.

\vspace{2mm}

The goal of this paper is to develop a Kazhdan-Lusztig theory for finite $W$-superalgebras of type $A$ in great generality. That is, we shall formulate and establish a canonical basis character formula for the irreducible modules in arbitrary parabolic BGG-type categories for finite $W$-superalgebras of type $A$, including all finite-dimensional irreducibles. Our formula provides a uniform generalization of several character formulas in BGG categories for Lie superalgebras and for $W$-algebras of type $A$. Our formula is new in the super setting beyond the BGG category for $\glnm$, and new in the non-super setting beyond the category of finite-dimensional $U(\gl_n,e)$-modules. We provide more detailed comments on the relevant works \cite{BK08, Los12, CCM23, CW25} in \S \ref{ssec:related}.

\subsection{Dual canonical and standard bases}

Given a partition $\la$ and a sign $\ep$, we have the associated $q$-symmetric powers $S^\la(\VV^\ep)$ in \eqref{eq:Sla}.
We also define the $q$-exterior powers $\wedge^{\la'} (\VV^\ep)$ and a polynomial representation $P^\la(\VV^\ep)$ which fit into the following sequence of homomorphisms (see \eqref{def:xilambda}--\eqref{EPS:V}; cf. \cite{Bru06,BK08}):
\begin{align*}
    \wedge^{\la'} (\VV^\ep) \twoheadrightarrow P^\la(\VV^\ep) \hookrightarrow S^{\la} (\VV^\ep).
\end{align*}
A signed multi-partition $\bflaep$ is a pair $(\bfla,\bfep)$ with a multi-partition $\bfla =(\la^{(1)},\ldots, \la^{(r)})$ and a sign sequence $\bfep=(\ep_1,\ldots,\ep_r)$, where $\ep_i=\pm$. Associated with $\bflaep$, we define
\[
S^{\bflaep}(\VV) =S^{\la^{(1)}} (\VV^{\ep_1}) \otimes \ldots \otimes S^{\la^{(r)}} (\VV^{\ep_r}),
\qquad
P^{\bflaep}(\VV) =P^{\la^{(1)}} (\VV^{\ep_1}) \otimes \ldots \otimes P^{\la^{(r)}} (\VV^{\ep_r}).
\]
We have a natural embedding $\jmath: P^{\bflaep}(\VV)\hookrightarrow S^{\bflaep}(\VV).$

Recall from \eqref{def:NA} the basis $\{\Pi_\bfA\mid \bfA \in \Row(\bflaep)\}$ for $S^{\bflaep}(\VV)$  and from \eqref{def:DeltaA} (also see \eqref{def:VA}) the standard basis $\{\Delta_\bfA\mid \bfA \in \Std(\bflaep)\}$ for $P^{\bflaep}(\VV)$. The following is a multi-partition generalization of some constructions from \cite{Bru06}, where the hat denotes a suitable completion. (As in loc.~cit., our constructions admit finite rank variants and then the hat can be removed.) As completed tensor product $\U$-modules, $\widehat S^{\bflaep}(\VV)$ and $\widehat P^{\bflaep} (\VV)$ admit bar involutions. Lusztig \cite{Lus92, Lus10} constructed (dual) canonical bases on tensor product modules more generally and somewhat differently. The standard basis $\{\Delta_\bfA\}$, which does not appear in Lusztig's construction, plays a crucial role here and admits a categorification.

\begin{alphatheorem} [Theorem \ref{thm:DCB:Slambda}, Theorem \ref{thm:DCB_P}, Proposition \ref{prop:sameDCB}]
\label{thm:A}
\qquad
\begin{enumerate}
    \item
    There exists a dual canonical basis $\{L_\bfA\mid \bfA\in \Row(\bflaep)\}$ on $\widehat S^{\bflaep}(\VV)$, which is characterized by bar invariance $\overline{L_\bfA}=L_\bfA$ and
   $L_\bfA \in \Pi_\bfA +\sum_{\bfA'\prec_{\bfs_{\bflaep}} \bfA} q^{-1}\Z[q^{-1}] \Pi_{\bfA'}.$
    \item
    There exists a dual canonical basis $\{L_\bfA \mid \bfA\in \Std(\bflaep)\}$ for $\widehat P^{\bflaep} (\VV)$, which is characterized by $\ov{L_\bfA}=L_\bfA$ and
   $L_\bfA \in \Delta_\bfA + \sum_{\bfA'\prec_{T,\bfep} \bfA} q^{-1} \Z[q^{-1}] \Delta_{\bfA'}.$
   \item
   Dual canonical bases are compatible under the embedding $\jmath:\widehat P^{\bflaep} (\VV) \hookrightarrow \widehat S^{\bflaep}(\VV)$.
\end{enumerate}
\end{alphatheorem}

\subsection{Simple character formulas}

Given a signed multi-partition $\bflaep$ of $(n|m)$-type, we attach an even nilpotent element in Lie superalgebra $\g=\glnm$ of Jordan type $(\la^+|\la^-)$; see \eqref{def:la_plusminus}. Different sign sequences $\bfep$ correspond to different choices of non-conjugate Borel subalgebras in $\glnm$; cf. \cite[\S 1.3.2]{CW12}. Consider the corresponding $W$-superalgebra $\Wg$ and form a parabolic category $\Pcat$ of $\Wg$-modules, following Losev \cite[Section 3]{Los12} and Brown-Goodwin \cite[Section 3]{BG13}.

Brundan-Kleshchev \cite{BK08} obtained a character formula for the finite-dimensional irreducible $U(\gl_n,\la)$-modules via dual canonical basis on the polynomial $\U$-module $P^\la(\VV)$, where finite-dimensional standard modules introduced in their paper play an important role. Standard modules $\Delta(\bfA)$ in $\Pcat$ are defined in \eqref{module:DeltaA} to be parabolically induced modules from tensor products of finite-dimensional standard modules, and they are not parabolic Verma modules in general. Denote by $[\Pcat]$ the Grothendieck group of $\Pcat$ and the hat denotes some suitable completion.

\begin{alphatheorem} [Theorem \ref{thm:char}]
\label{thm:B}
\qquad
 \begin{enumerate}
     \item
There is a $\Z$-linear isomorphism
$%    \begin{align*}
    \Psi: [\Pcat]^\wedge \rightarrow \widehat P^{\bflaep} (\VV_\Z),
$
%    \end{align*}
    which sends the classes of standard modules $[\Delta(\bfA)]$ to the standard basis $\Delta_\bfA$, for $\bfA \in \Std(\bflaep)$;
    \item
    The map $\Psi$ sends the classes of simple modules to the dual canonical basis, i.e., $\Psi([L(\bfA)]) = L_\bfA$, for $\bfA \in \Std(\bflaep)$.
\end{enumerate}
\end{alphatheorem}

We outline the strategy of a proof of Theorem \ref{thm:B}. A signed multi-partition $\bflaep$ gives rise to a refined signed multi-partition $\uep\,\ula$ in \eqref{def:ula_uep}. The category $\calP(\ula,\uep)$ associated with $\uep\,\ula$ in \eqref{def:ula_uep} is the full category of $\Wg$-modules denoted by $\Ofull$, and we have an identification of $\U$-modules $S^{\bflaep}(\VV) = P^{\uep\,\ula}(\VV)$. We first establish Theorem \ref{thm:B} for $\Ofull$ by applying $\bfep$-variants of the main results from \cite{CCM23} and \cite{CW25} which pass through categories of Whittaker modules; the formulations in loc.~cit.~treated the case where $\bfep$ is the standard sign sequence with $+$'s followed by $-$'s.

Theorem \ref{thm:B} in full generality follows from this full category $\calO$ version and Theorem~ \ref{thm:A}, once we establish the following commutative diagram:
\begin{equation*}
\begin{tikzcd}
{[}\Pcat{]}^\wedge \ar[r,"\jmath_o"]\ar[d,"\Psi"]&{[}\Ofull{]}^\wedge \ar[d,"\Psi\!_o"]\\
\widehat P^{\bflaep} (\VV_\Z) \ar[r,"\jmath"]& \widehat S^{\bflaep} (\VV_\Z)
\end{tikzcd}
\end{equation*}
where $\jmath$ and $\jmath_o$ are natural embeddings.

The main ingredient for proving the above commutative diagram is the following character identity for the classes of standard modules $\Delta(\bfA)$ in terms of another family of modules denoted by $N(\bfA)$; recall $N(\bfA)$ is defined in \eqref{def:moduleNA} as the pullback of a parabolic Verma module over the Levi via Miura transform.

\begin{alphatheorem} [Theorem \ref{thm:charStandard}]
\label{thm:C}
  The following identity holds in  $[\Pcat]$:
\[
[\Delta(\bfA)] =[N(\bfA)], \qquad \text{ for } \bfA\in\Std(\bflaep).
\]
\end{alphatheorem}
We derive Theorem \ref{thm:C} from its special case -- the corresponding character identity for Verma modules $M(B)$ in $\Ofull$. This Verma character identity in the non-super setting is a main technical result of Brundan-Kleshchev \cite[Corollary 6.3]{BK08}, and its highly nontrivial super generalization is obtained by Lu-Peng \cite{LP26}. Note that \cite[Corollary~ 6.3]{BK08} is a consequence of (and is actually equivalent to) \cite[Theorem 6.2]{BK08} which provides a Gelfand-Tsetlin character formula for the Verma modules via the connection between finite $W$-algebras and Yangians.

In the case where $\bfla=(\la,\mu)$ and $\bfep=(+,-)$, where $\la$ and $\mu$ are partitions of $n$ and $m$, denote the (maximal) parabolic category $\Pcat$ by $\calF(\la|\mu)$. The category $\calF (\la|\mu)$ can be identified as the category of finite-dimensional $\Wgl$-modules of integer weights.

\begin{alphatheorem} [Theorem \ref{thm:fd_Char}]
\label{thm:D}
Let $\bfla=(\la,\mu)$ and $\bfep=(+,-)$.
Then the $\Z$-module isomorphism $\Psi: [\calF(\la|\mu)]^\wedge \rightarrow
    P^\la(\VV_\Z) \widehat\otimes P^\mu(\VV^-_\Z)$, defined by sending $\Delta(A^{(1)},A^{(2)}) \mapsto \Delta_{(A^{(1)},A^{(2)})}$, maps  $[L(A^{(1)},A^{(2)})]$ to $L_{(A^{(1)},A^{(2)})}$.
\end{alphatheorem}

\subsection{Relations to other works}
\label{ssec:related}

Several very distinguished cases of Theorem~\ref{thm:B} have been established in the literature in the past decades. The statements in these cases are powerful theorems in their own right and there was no a priori indication that they share a vast uniform generalization in the finite $W$-superalgebra setting as formulated in Theorem \ref{thm:B}.

In the non-super setting, notation gets much simplified as the sign sequences and $\VV^-$ are not needed. Theorem~\ref{thm:B} for finite $W$-algebra $U(\gl_n,\la)$ is new, with the following exceptions. In the case when the multi-partition $\bfla$ is a partition $\la$, the category reduces to the category of {\em finite-dimensional} $U(\gl_n,\la)$-modules, and the character formula in Theorem \ref{thm:D} was a main result of \cite[Theorem~F]{BK08}. In another extreme case when $\bfla=\ula$ and the parabolic category is the full category $\calO(\la)$, we observe that the character formula in Theorem \ref{thm:B} follows by a combination of the following results:

$\triangleright$ a character formula of  Mili\v{c}i\'{c}-Soergel (MS) \cite{MS97} for simple Whittaker modules;

$\triangleright$ a reformulation \cite{CCM23} via dual canonical basis in $S^\la(\VV)$ of MS character formula;

$\triangleright$ a category equivalence \cite{Los12} between $\calO(\la)$ and a category of Whittaker modules.

This proves \cite[Conjecture~ 7.17]{BK08}; also cf. \cite{VD95}.
For $\la=(1^n)$, Theorem \ref{thm:B} further specializes to the classic theorems \cite{BB81, BK81} on Kazhdan-Lusztig conjecture in type $A$.

In the super setting (with both $n, m$ nonzero), Theorem~\ref{thm:B} was known in a distinguished case. In case that all partitions in $\bfla$ are of one-column,  Theorem~\ref{thm:B} reduces to the celebrated Brundan-Kazhdan-Lusztig conjecture for $\glnm$ for arbitrary Borel and its parabolic variants \cite{Bru03, CLW15,BLW17}, relying on canonical bases of $\wedge\!^\la\VV\otimes \wedge\!^\mu\VV^{-}$ for partitions $\la$ and $\mu$ of $n$ and $m$. In another extreme case where the parabolic category is the full category $\Ofull$ associated with a standard sign sequence, the character formula is built on a canonical basis of $S^\la\VV\otimes S^\mu\VV^{-}$; we refer to Theorem \ref{thm:Char_fullO} and its proof where the works \cite{CCM23, CW25, Los12, SX20} are crucially used. The canonical basis character formula in Theorem \ref{thm:B} ultimately relies on Brundan-Kazhdan-Lusztig character formula for BGG category of $\glnm$-modules.

From the tensor product module formulation in Theorem \ref{thm:B}, it is tempting to relate it to tensor product categorification \cite{WebMem}, which works with standardly stratified categories. The category $\Ofull$ is rarely standardly stratified; however, it even admits a properly stratified subcategory $\Ofull^{\text{ps}}$ which contains Verma modules and the same irreducibles as $\Ofull$, via constructions of cokernel categories; cf. \cite{CCM23} (which treated the special case of a standard sign sequence but is valid for any $\bfep$) and the references therein. The uniqueness theorem from \cite{LW15} implies an equivalence of categories between $\Ofull^{\text{ps}}$ and Webster's diagrammatic tensor product category. For general parabolic categories $\Pcat$, it is an open problem on how to produce properly stratified categories with analogous favorable properties as the cokernel category approach does not seem to apply.

It will be interesting to formulate a Jantzen filtration on Verma modules for finite W-(super)algebras. It is natural to conjecture that Jantzen filtration in type A provides a $q$-analogue of our character formula in terms of the dual canonical basis in Theorem \ref{thm:A}.

In a future work, we will generalize this work to type $BCD$, where Bao-Wang's $\imath$canonical basis will be used instead of Lusztig's canonical basis in type $A$.

\subsection{Organization}

In Section \ref{sec:bases}, we construct standard bases, dual canonical bases, and related Bruhat orderings on various modules of the quantum group $\U$, including $q$-symmetric tensor, $q$-skew-symmetric tensor, polynomial modules, and their tensor product modules.
In Section \ref{sec:character}, we formulate the BGG-type categories $\Pcat$ for finite $W$-superalgebras of type $A$. We establish a character identity for the standard modules  and the character formula for the irreducible modules in $\Pcat$ using the dual canonical basis constructed in the previous section.

%\begin{enumerate}
%\item Koszulity, parabolic-singular duality ?
%\item In \cite{VD96}, the BRST 0th cohomology provides a functor from BGG category of $\g$-modules to ??category $\mc O$?? of $W$-modules. The KL conjecture for $W$ follows under the assumption of higher BRST cohomology vanishing.

%Is this BRST functor related or equal to the composition of Losev isomorphism with Backlin functor?
%\end{enumerate}

\vspace{2mm}
\noindent{\bf Acknowledgments.}
The first author is partially supported by an NSTC grant of the R.O.C. The second author is partially supported by the NSF grant DMS-2401351. He thanks Department of Mathematics and IMS, National University of Singapore,  and Institute for Mathematics, Academia Sinica in Taipei, and the NSTC for providing excellent research environment and for support during his sabbatical leave. We thank Chih-Whi Chen for helpful discussion on cokernel categories.

\section{Dual canonical and standard bases on tensor product modules}
\label{sec:bases}

In this section, we study the $q$-multilinear algebra (including $q$-symmetric tensors and $q$-anti-symmetric tensors) in the framework of representations of the quantum group of type $A$. We construct the standard bases and dual canonical bases in various tensor product modules. We develop the combinatorics of multi-tableaux and Bruhat orderings which provide parameterizations of the tensor product modules.

\subsection{Tableaux}
\label{ssec:tab}

Let $n\in \N$. Associated to a partition $\la$ of $n$, we have a left-justified pyramid (or upside down Young diagram) of shape $\la$. For example, the pyramid form of $\la=(3,2,2)$ is as follows:

    \begin{align*}
   \ydiagram{2,2,3}
    \end{align*}

A $\la$-tableau is a filling of the boxes of $\la$ with integers.
For $\la$ of length $\ell(\la)=\ell$, we shall always number the rows of a $\la$-tableau $A$ by $1, \ldots, \ell$ from top to bottom and the columns by $1,\ldots, \wp$ from left to right. The sequence $c(A) \in \Z^n$ is obtained from $A$ by column reading the entries down the columns starting from the leftmost column. Writing $c(A) =(a_1,\ldots,a_n)$, we can assign to $A$ a weight
\begin{align} \label{wtP}
\wt(A):= \sum_{i=1}^n \delta_{a_i}
 \in
\texttt{P},
\qquad \text{where }
\texttt{P} :=\bigoplus_{a\in \Z} \Z\delta_a.
\end{align}
On the other hand, we define a sequence $\rho(A)$ by row reading the entries of $A$ along the rows from left to right starting from the top row of $A$.

Denote by $\Row(\la)$  the set of $\la$-tableaux that are weakly increasing along each row from left to right. Denote by $\Col(\la)$ the set of column strict $\la$-tableaux, i.e., the tableaux that are strictly increasing along each column from bottom to top.
%An $A \in \Row(\la)$ is {\em dominant} if it has a representative belonging to $\Col(\la)$ and let $\Dom(\la)$ denote the set of all such dominant row symmetrized $\la$-tableaux.
A $\la$-tableau $A$ is {\em standard} if the entries of $A$ are strictly increasing up columns from bottom to
top and weakly increasing along rows from left to right. Denote by $\Std(\la)$ or $\Std(\la,+)$ the set of all standard $\la$-tableaux. Our convention here is largely consistent with \cite[\S2]{Bru06}.

We shall need reverse versions of various notions of tableaux defined above, which will be denoted by $\Row(\mu,-), \Col(\mu,-), \Std(\mu,-)$, for a partition $\mu$ of $n$. More explicitly, we denote by $\Row(\mu,-)$  the set of $\mu$-tableaux which are weakly {\em decreasing} along each row from left to right.
Denote by $\Col(\mu,-)$ the set of reverse column strict $\mu$-tableaux, i.e., the tableaux which are strictly {\em decreasing} up on each column from bottom to top.
%An $A \in \Row({\mu})$ is reverse dominant if it has a representative belonging to $\Col(\mu,-)$ and let $\Dom(\mu,-)$ denote the set of all such reverse dominant row symmetrized $\mu$-tableaux.
A $\mu$-tableau $B$ is {\em reverse standard} if the entries of $A$ are strictly {\em decreasing} up columns from bottom to
top and weakly {\em decreasing} along rows from left to right. Denote by $\Std(\mu,-)$ the set of all reverse standard $\mu$-tableaux.
%The rectification map $R: \Std(\la) \rightarrow \Dom(\la)$, which sends $A$ to its row equivalence class, is a bijection.
Writing $c(A) =(a_1,\ldots,a_n)$, we assign to $A \in\Std(\mu,-)$ the weight $-\wt(A)$; cf. \eqref{wtP} for notation $\wt(A)$.

Summarizing, for $A\in \Std(\la, \ep)$ with $\ep=\pm$, we have uniformly defined its weight to be $\wt_\ep(A) :=\ep \wt(A)$.

Let $\bfla =(\la^{(1)}, \la^{(2)},\ldots, \la^{(r)})$ be a multi-partition with each $\la^{(i)}$ being a partition, and let $\bfep=(\ep_1,\ep_2,\ldots,\ep_r)$ be an associated sign sequence with each $\ep_i\in \{\pm\}$. We shall call the pair $(\bfla,\bfep)$ (or simply denote by $\bflaep$ when a compact notation is preferred) a signed multi-partition. We associate to a signed multi-partition $\bflaep$ the following sets of signed multi-tableaux:
\begin{align}\label{Tab:bfla}
\begin{split}
\Row(\bflaep) &=\prod_{i=1}^r\Row(\la^{(i)},\ep_i),
\\
\Col(\bflaep) &=\prod_{i=1}^r\Col(\la^{(i)},\ep_i),
\\
\Std(\bflaep) &=\prod_{i=1}^r\Std(\la^{(i)},\ep_i).
\end{split}
\end{align}
For $n,m\in\N$, a sign sequence $\bfep$ is called {\em an $(n|m)$-sequence} if it contains $n$ pluses and $m$ minuses; it is {\em standard} if it is of the form $(+^n,-^m)$.

\begin{ex} \label{ex:tabeaux}
    Let $\la$ be the partition $(3,2,2)$.
Consider the following three $\la$-tableaux:
\begin{align*}
 \begin{ytableau}
%    4 & \none & \none \\
    2 & 2  \\
    3 & 6  \\
    0 & 0 & 1
\end{ytableau}\qquad\qquad
   \begin{ytableau}
    3 & 6  \\
    2 & 2  \\
    0 & 1 & 0
\end{ytableau}   \qquad\qquad
   \begin{ytableau}
    3 & 6  \\
    2 & 2  \\
    0 & 0 & 1
\end{ytableau}
\end{align*}
The first one lies in $\Row(\la,+)$, the second one in $\Col(\la,+)$, while the third one is standard. On the other hand the following three tableaux lie in $\Row(\la,-)$, $\Col(\la,-)$, and $\Std(\la,-)$, respectively:
\begin{align*}
 \begin{ytableau}
    2 & 2  \\
    1 & 0  \\
    3 & 3 & 6
\end{ytableau}\qquad\qquad
   \begin{ytableau}
    1 & 0  \\
    2 & 2  \\
    3 & 3 & 6
\end{ytableau}   \qquad\qquad
   \begin{ytableau}
    1 & 0  \\
    2 & 2  \\
    6 & 3 & 3
\end{ytableau}
\end{align*}
\end{ex}

\subsection{Bruhat orderings}
 \label{super:Bruhat}

Let $k\in \N$ and $\ep=\pm$. We identify $f\in\Z^k$ with the $k$-tuple of integers $\left(f(1), \ldots,f(k)\right)$. We call $f$ {\em $\ep$-dominant} if
\begin{align}
 \label{eq:Zk+}
\begin{cases}
f(1)> f(2)>\cdots> f(k), & \text{ if } \ep=+,
\\
f(1)< f(2)< \cdots < f(k), & \text{ if } \ep=-.
\end{cases}
\end{align}
Denote by $\Z^{k,\ep}_+$ the subset of $\ep$-dominant elements in $\Z^k$.
We call $f\in\Z^k$ {\em $\ep$-anti-dominant} if
\begin{align}
\label{eq:Zk-}
\begin{cases}
f(1)\leq f(2)\leq\cdots\leq f(k), & \text{ if } \ep=+,
\\
f(1)\geq f(2)\geq \cdots \geq f(k), & \text{ if } \ep=-.
\end{cases}
\end{align}
Denote by $\Z^{k,\ep}_-$ the subset of $\ep$-anti-dominant elements in $\Z^k$.
The natural bijections $\Z^k_+ \leftrightarrow\Col((1^k),\ep)$ and $\Z^k_- \leftrightarrow\Row((k),\ep)$ will be generalized later.

Let $\bfs=(s_1,\ldots, s_{n+m})$ be an $(n|m)$-sign sequence.
We introduce a partial order on $\Z^{n+m}$, $\preceq_\bfs$, that is compatible with the Bruhat orderings for the Lie superalgebra $\glnm$; cf. \cite{CLW15}.

Define the following subset of $P$ from \eqref{wtP}:
\begin{align}  \label{P+}
\texttt{P}^+ =\bigoplus_{a\in \Z} \N (\delta_a -\delta_{a+1}).
\end{align}
We define a partial order on $\texttt{P}$ by
declaring $\nu\ge\mu$, for $\nu,\mu\in \texttt{P}$, if $\nu-\mu \in \texttt{P}^+$.

For $f\in\Z^{n+m}$ and $1\le j\le n+m$, we define
\begin{align*}
\wt^j_{\bfs}(f)
 :=\sum_{i=j}^{n+m} {s_i}\delta_{f(i)}\in \texttt{P},\quad
\wt_{\bfs}(f):=\wt^{1}_{\bfs}(f)\in \texttt{P}.
\end{align*}

We define the {\em Bruhat ordering} on $\Z^{n+m}$,
denoted by $\preceq_{\bfs}$, in terms of the partially ordered set
$(\texttt{P}, \leq)$ as follows: $g\preceq_{\bfs}f$ if and only if
$\wt_{\bfs}(g)=\wt_{\bfs}(f)$ and $\wt^j_{\bfs}(g)\le\wt^j_{\bfs}(f)$, for all $j$. Note that when $m=0$ and $\bfs=(+^n)$
this is simply the usual Bruhat ordering on the weight lattice
$\Z^n$ of $\gl_n$.

%Introduce the following sequence of integers, for $f\in\Z^{n+m}, a\in \Z$ and $1\le j \le n+m$:
%
%\begin{align}\label{aux:sharp:Bruhat}
%\sharp_{\bfs}(f,a,j):=\sum_{j\le i\leq n+m,f(i)\le a} {s_i}  1.
%\end{align}
%Then the following characterization of $\preceq_\bfs$ holds:
%
%\begin{align*}
%&g\preceq_{\bfs}f, \; \text{ for } g,f\in\Z^{n+m} \; \Leftrightarrow\;
%\\
%&\sharp_{\bfs}(g,a,j)\le \sharp_{\bfs}(f,a,j),\; \forall a\in\Z,1\le j\le n+m, \text{with equality for $j=1$.}
%\end{align*}
%
%For $1\le i\le n+m$ we let $d_i\in\Z^{n+m}$ be determined by
%
%\begin{equation}  \label{eq:di}
%d_i(j)= {s_i}\delta_{ij}, \quad \text{ for } j\in [n+m].
%\end{equation}
%Let $f,g\in\Z^{n+m}$. We write $f\downarrow_{\bfs} g$ if one of the following holds:
%\begin{align*}
%\begin{cases}
%g=f\cdot (i,j),&\text{ for }s_i=s_j=+,i<j, f(i)>f(j),\\
%g=f\cdot (i,j),&\text{ for }s_i=s_j=-,i<j, f(i)<f(j),\\
%g=f-d_i+d_j,&\text{ for }s_i\not=s_j, i<j, f(i)=f(j).
%\end{cases}
%\end{align*}
%Here $f\cdot (i,j)$ denotes the right action by a transposition.

\subsection{Tensor $\U$-modules $\bbT^{\bfs}$}

Let $q$ be an indeterminate. The quantum group $\U =\U_q({\mathfrak
g\mathfrak l}_\infty)$ is defined to be the associative algebra over
$\Q(q)$ generated by $E_a, F_a, K_a, K^{-1}_a, a \in \Z$, subject to
the following relations ($a,b\in\Z$):
\begin{eqnarray*}
 K_a K_a^{-1} &=& K_a^{-1} K_a =1, \quad
 K_a K_b = K_b K_a, \\
 K_a E_b K_a^{-1} &=& q^{\delta_{a,b} -\delta_{a,b+1}} E_b, \\
 K_a F_b K_a^{-1} &=& q^{\delta_{a,b+1}-\delta_{a,b}}
 F_b, \\
 E_a F_b -F_b E_a &=& \delta_{a,b} \frac{K_{a,a+1}
 -K_{a+1,a}}{q-q^{-1}}, \\
 E_a^2 E_b +E_b E_a^2 &=& (q+q^{-1}) E_a E_b E_a,  \quad \text{if } |a-b|=1, \\
 E_a E_b &=& E_b E_a,  \,\qquad\qquad\qquad \text{if } |a-b|>1, \\
 F_a^2 F_b +F_b F_a^2 &=& (q+q^{-1}) F_a F_b F_a,  \quad\, \text{if } |a-b|=1,\\
 F_a F_b &=& F_b F_a,  \qquad\ \qquad\qquad \text{if } |a-b|>1.
\end{eqnarray*}
Here $K_{a,a+1} :=K_aK_{a+1}^{-1}$. For $r\ge 1$, we introduce the
divided powers $E_a^{(r)} =E_a^{r}/[r]!$ and
$F_a^{(r)}=F_a^{r}/[r]!$, where $[r] =(q^r-q^{-r})/(q-q^{-1})$.

Setting $\ov{q}=q^{-1}$ induces an automorphism on $\Q(q)$ denoted
by $^{-}$. Define the bar involution on $U_q({\mathfrak g\mathfrak
l}_\infty)$ to be the anti-linear automorphism ${}^-: U_q({\mathfrak
g\mathfrak l}_\infty)\rightarrow U_q({\mathfrak g\mathfrak
l}_\infty)$ determined by $\ov{E_a}= E_a$, $\ov{F_a}=F_a$, and
$\ov{K_a}=K_a^{-1}$. Here {\em anti-linear} means that
$\ov{fu}=\ov{f}\ov{u}$, for $f\in\Q(q)$ and $u\in U_q(\gl_\infty)$.
The quantum group $U_q(\mf{sl}_\infty)$ is the subalgebra generated
by $\{E_a,F_a,K_{a,a+1}\vert a\in\Z\}$.

Set $\cZ=\Z[q,q^{-1}]$ and let $\U\!_\cZ=\U (\sln_\infty)\!_\cZ$ denote the Lusztig integral form quantum group, i.e., the $\cZ$-subalgebra generated by the divided powers $E_i^{(r)}, F_i^{(r)}$, for $i\in \Z$ and $r\ge 0$.

Let $\VV=\VV^+$ be the natural $U_q({\mathfrak g\mathfrak
l}_\infty)$-module with basis $\{v_a\}_{a\in\Z}$ and let $\VV^-$ be the restricted dual module with basis
$\{v^-_a\}_{a\in\Z}$ such that $$\langle v^-_a ,v_b\rangle = (-q)^{-a}
\delta_{a,b}.$$  The actions of $U_q(\gl_\infty)$ on $\VV$ and $\VV^-$ are given by the following formulas:
\begin{align*}
K_av_b=q^{\delta_{ab}}v_b,\qquad E_av_b&=\delta_{a+1,b}v_a,\ \ \quad
F_av_b=\delta_{a,b}v_{a+1},\\
K_av^-_b=q^{-\delta_{ab}}v^-_b,\quad E_av^-_b&=\delta_{a,b}v^-_{a+1},\quad
F_av^-_b=\delta_{a+1,b}v^-_{a}.
\end{align*}
We also denote $v^+_a:=v_a$ so we can write $v_a^\ep$, for a sign $\ep$. Then $\VV^\ep_\cZ:=\sum_{a\in\Z}\cZ v_a^\ep$ is a $\U_{\cZ}$-submodule of $\VV^\ep$.

We shall use the co-multiplication $\Delta$ on $U_q(\gl_\infty)$
defined by
\begin{align}  \label{eq:Delta}
\begin{split}
 \Delta (E_a) &= 1 \otimes E_a + E_a \otimes K_{a+1, a}, \\
 \Delta (F_a) &= F_a \otimes 1 +  K_{a, a+1} \otimes F_a,\quad
 \Delta (K_a) = K_a \otimes K_{ a},
 \end{split}
\end{align}
which restricts to a comultiplication on $U_q(\mf{sl}_\infty)$. Our
$\Delta$ here
differs from the one used in \cite{Lus10}.

Let $n,m\in\N$. Given an $(n|m)$-sign sequence $\bfs =(s_1, \ldots, s_{n+m})$, we define the following
tensor space over $\bbQ(q)$:
\begin{equation*}
{\bbT}^\bfs = {\bbT}^\bfs(\VV):= {\mathbb V^{s_1}} \otimes {\mathbb V^{s_2}} \otimes \ldots \otimes {\mathbb V^{s_{n+m}}}.
\end{equation*}
For example, associated to the standard sign sequence $\bfs\bft =(+^n,-^m)$, we have ${\bbT}^{\bfs \bft} =\VV^{\otimes n}\otimes {\VV^-}^{\otimes m}$; in case $m=0$ or $n=0$ we will simply write as ${\bbT}^{n}(\VV)$ and ${\bbT}^{m}(\VV^-)$.
The quantum group $U_q(\gl_\infty)$ acts on $\bbT^\bfs$ via iterated co-multiplication.

For $f\in \Z^{n+m}$, we define
\begin{equation}  \label{eq:Mf}
M_f =M_f^\bfs := v_{f(1)}^{s_1}\otimes v_{f(2)}^{s_2}\otimes\cdots \otimes v_{f(n+m)}^{s_{n+m}}.
\end{equation}
We refer to $\{M_f \mid f\in
\Z^{n+m}\}$ as the {\em monomial basis} for ${\mathbb
T}^\bfs$. Denote by ${\bbT}^\bfs_\cZ$ the $\Z[q,q^{-1}]$-span of $\{M_f \mid f\in
\Z^{n+m}\}$.

Denote by $\widehat{\bbT}^\bfs$ the (downward) completion of $\bbT^\bfs$ with respect to the monomial basis in Bruhat order $\preceq$, i.e., its elements are finite linear combinations of possibly infinite sums of the form $M_f +\sum_{f'\prec f} c_{f,f'}(q) M_{f'}$, for $c_{f,f'} \in \Q(q)$. We define $\widehat{\bbT}^\bfs_\cZ$ similarly. We will encounter several similar (downward) completion below.

There is a bar involution $\overline{\phantom{x}}:\widehat{\bbT}^\bfs\rightarrow \widehat{\bbT}^\bfs$ defined via quasi-R matrix (cf.~\cite{Lus92, Bru03, CLW15}) such that
$\overline{M}_f\in M_f+\sum_{g\prec_{\bfs} f}\Z[q,q^{-1}] M_g.$
It follows from \cite[Lemma 24.2.1]{Lus10} that there exists dual canonical basis $\{\pounds_f \mid f\in\Z^{n+m}\}$ in $\widehat{\bbT}^\bfs$, which is characterized by $\ov{\pounds}_f=\pounds_f$ and the property
\begin{align}\label{dual:can:T}
    \pounds_f\in M_f+ \sum_{g\prec_{\bfs} f}q^{-1}\Z[q^{-1}] M_g.
\end{align}

\subsection{$q$-symmetric tensors}\label{ssec:qsymm}

The symmetric group $\mf{S}_{k}$ on $\{{1},{2},\ldots,{k}\}$ is generated
by the simple transpositions $s_{i} =(i,i+1)$, for $1\le i\le k-1$.
The Hecke algebra $\mathcal H_{k}$ is the associative $\mathbb Q(q)$-algebra
 generated by $H_i$, $1 \le i \le k-1$, subject to
the relations $(H_i -q^{-1})(H_i +q) = 0$ and the usual braid relations.
%\begin{align*}
%&(H_i -q^{-1})(H_i +q) = 0,\\
%&H_i H_{i+1} H_i = H_{i+1} H_i H_{i+1},\\
%&H_i H_j = H_j H_i, \quad\text{for } |i-j| >1.
%\end{align*}
Associated to $\sigma \in \mf{S}_{k}$ with a reduced
expression $\sigma=s_{i_1} \cdots s_{i_t}$, we define the element
 $
H_\sigma :=H_{i_1} \cdots H_{i_t}. $ The bar involution $\overline{\phantom{x}}$ on $\mathcal
H_{k}$ is the unique anti-linear automorphism such that
$\overline{q} =q^{-1}$ and $\overline{H_\sigma} =H_{\sigma^{-1}}^{-1}$,
for all $\sigma \in \mf{S}_{k}$.

Consider the $k$-th tensor space $\bbT^k(\VV^\ep)$ of $\VV^\ep$. Recall from \eqref{eq:Mf} that
$M_f = v^\ep_{f(1)} \otimes \cdots \otimes
v^\ep_{f(k)}$, for $f\in\Z^k$. The algebra $\mathcal H_{k}$ acts on $\bbT^k(\VV^\ep)$ on the right by
\begin{eqnarray} \label{eq:heckeaction}
 M_f H_i = \left\{
 \begin{array}{ll}
 M_{f\cdot s_i}, & \text{if } f \prec f \cdot s_i,  \\
 q^{-1} M_{f}, & \text{if } f = f \cdot s_i, \\
 M_{f \cdot s_i} - (q-q^{-1}) M_f, & \text{if } f \succ f \cdot
 s_i.
 \end{array}
 \right.
\end{eqnarray}
Here $f\cdot s_i\prec f$ in  case $\ep=+$ while $f\prec f\cdot s_i$ in case $\ep=-$, if $f(i)>f(i+1)$.
By the Jimbo-Schur duality (\cite{Jim86}), the actions of $U_q(\gl_\infty)$ and $\mathcal H_{k}$ on the tensor space $\bbT^k(\VV^\ep)$ commute with each other. Set
\begin{align}
\begin{split}
\texttt{Sym}_k &:= \sum_{\sigma \in \mf{S}_{k}} q^{\ell(w_k)-\ell(\sigma)} H_\sigma,
 \qquad
 \texttt{Ant}_k := \sum_{\sigma \in \mf{S}_{k}}
(-q)^{\ell(\sigma)-\ell(w_k)} H_\sigma,
\end{split}
\label{formula:Y0}
\end{align}
where $w_k$ denotes the longest element in $\mf{S}_k$. It is
well known that $\texttt{Sym}_k$ and $\texttt{Ant}_k$ are bar invariant, and moreover,
\begin{align}\label{eq:S0H}
\begin{split}
\texttt{Sym}_k H_\sigma &= q^{-\ell(\sigma)} \texttt{Sym}_k = H_\sigma \texttt{Sym}_k,
  \\
\texttt{Ant}_k H_\sigma &= (- q)^{\ell(\sigma)} H_\sigma = H_\sigma \texttt{Ant}_k,
\end{split}\qquad \text{ for }\sigma \in \mf S_k.
\end{align}

Right multiplication by $\texttt{Sym}_k$ defines a $\U$-module homomorphism $\texttt{Sym}_k:\bbT^k(\VV^\ep)\rightarrow\bbT^k(\VV^\ep)$.
We define the $k$-th $q$-symmetric tensor of $\VV^\ep$ to be
\begin{align*}
    S^k(\VV^\ep):= \bbT^k(\VV^\ep)/\ker\texttt{Sym}_k.
\end{align*}
Indeed, $\ker\texttt{Sym}_k$ is the sum of the kernel of the $\U$-homomorphisms $H_i+q$, $1\le i\le k-1$. Since $\texttt{Sym}_k$ is bar invariant, it follows that the $\U$-action and the bar involution on $\bbT^k(\VV^\ep)$ descend to a bar involution on the $\U$-module $S^k(\VV^\ep)$. The monomial basis in $S^k(\VV^\ep)$ is by definition the images $\{\pi(M_f)|f\in\Z^{k,\ep}_-\}$ under the canonical map $\pi:\bbT^k(\VV^\ep)\rightarrow S^k(\VV^\ep)$; cf. \eqref{eq:Zk-} for $\Z^{k,\ep}_-$.

Thanks to the commuting actions of $\U$ and $\mc H_k$ on $\bbT^k(\VV^\ep)$, we see that
$\bbT^k(\VV^\ep)\texttt{Sym}_k$ is a $\U$-module. Since $\bbT^k(\VV^\ep)=\ker \texttt{Sym}_k\oplus\bbT^k(\VV^\ep)\texttt{Sym}_k$ as $\U$-modules, the projection $\pi: \bbT^k(\VV^\ep)\rightarrow S^k(\VV^\ep)$ leads to a $\U$-module isomorphism $\bbT^k(\VV^\ep)\texttt{Sym}_k \cong S^k(\VV^\ep).$
%
%Let $f\in \Z^{k,\ep}_-$. Denote its stabilizer by $W_f:=\{\sigma\in\mf{S}_k|f\cdot\sigma=f\}$ and denote by $w_k^f$ the longest element in $W_f$. We have $W_f=\prod_{i}\mf S_{m_i}$ and we set $[|W_f|]:=\prod_{i}[m_i]!$. Furthermore, let $W^f$ denote the set of shortest length representatives of the cosets $W_f\backslash\mf{S}_k$ with longest element $^f\!w_k$. We compute
%\begin{align*}
%M_f \texttt{Sym}_k&=M_f \sum_{\sigma\in W_f}q^{\ell(w_k^f)-\ell(\sigma)}H_\sigma\sum_{\tau\in W^f}q^{\ell(^f\!w_k)-\ell(\tau)}H_\tau\\
%&=M_f \sum_{\sigma\in W_f}q^{\ell(w_k^f)-2\ell(\sigma)}\sum_{\tau\in W^f}q^{\ell(^f\!w_k)-\ell(\tau)}H_\tau\\
%&=[|W_f|]M_f \sum_{\tau\in W^f}q^{\ell(^f\!w_k)-\ell(\tau)}H_\tau\\
%&=[|W_f|]\sum_{\tau\in W^f}q^{\ell(^f\!w_k)-\ell(\tau)}M_{f\cdot \tau}.
%\end{align*}
%In the penultimate equality above, we have used the formula
%\begin{align*}
%\sum_{\sigma\in\mf S_k}q^{\ell(w_k)-2\ell(\sigma)}=[k]!.
%\end{align*}
%Thus,
%\begin{align*}
%\widetilde{M}_f:=\frac{1}{[|W_f|]} M_f \texttt{Sym}_k=\sum_{\tau\in W^f}q^{\ell(^f\!w_k)-\ell(\tau)}M_{f\cdot\tau}
%\end{align*}
%lies in the $Z[q,q^{-1}]$-span of the $M_f$. Then $\{\widetilde{M}_f|f\in\Z^{k,\ep}_-\}$ is a basis for $\widetilde{S}^k(\VV^\ep)$.
%
Let $S^k(\VV^\ep\!\!\!_\cZ)$ be the $\Z[q,q^{-1}]$-span of the monomial basis in $S^k(\VV^\ep)$. Furthermore, let  $\bbT^{k}(\VV^\ep\!\!\!_\cZ)\texttt{Sym}_k$ be $\Z[q,q^{-1}]$-span of $M_{f}\texttt{Sym}_k$, for $f\in\Z^{k,\ep}$. Then, the $\cZ$-linear map sending $\pi(M_f)\mapsto M_{f}\texttt{Sym}_k$, for $f\in\Z^{k,\ep}_-$, gives an isomorphism of $\U\!_{\cZ}$-modules $S^k(\VV^\ep\!\!\!_\cZ)\cong \bbT^{k}(\VV^\ep\!\!\!_\cZ)\texttt{Sym}_k$ with compatible bar involutions. The isomorphism matches the monomial basis $\{\pi(M_f)|f\in\Z^{k,\ep}_-\}$ of $S^k(\VV)$ with the monomial basis of $\{M_f\texttt{Sym}_k|f\in\Z^{k,\ep}_-\}$ of $\bbT\!_{\cZ}^{k}(\VV^\ep)\texttt{Sym}_k$.

For a composition $\la=(\la_1,\dots, \la_\ell)$ and a sign $\ep$, we let
\begin{align}
   \label{eq:Sla}
S^{\la} (\VV^\ep) =S^{\la_\ell} (\VV^\ep)   \otimes \cdots \otimes S^{\la_1} (\VV^\ep),
\quad
S^{\la} (\VV^\ep\!\!\!_\cZ) =S^{\la_\ell} (\VV^\ep\!\!\!_\cZ)   \otimes \cdots \otimes S^{\la_1} (\VV^\ep\!\!\!_\cZ).
\end{align}
We have a canonical map $\bbT^{|\la|} (\VV^\ep) \rightarrow S^{\la} (\VV^\ep)$.

For a signed multi-partition $\bflaep$ with $\bfla =(\la^{(1)},\ldots, \la^{(r)})$ and $\bfep=(\ep_1,\ldots,\ep_r)$, we define
compositions
\begin{align}
    \label{def:la_plusminus}
\la =\bigcup_{i=1}^r \la^{(i)},
\qquad
\la\!^+ =\bigcup_{i:\ep_i=+} \la^{(i)},
\qquad
\la\!^- =\bigcup_{i:\ep_i=-} \la^{(i)},
\end{align}
with parts of $\la^{(i)}$ preceeding parts of $\la^{(j)}$ whenever $i<j$. Writing $\la=(\la_1,\la_2,\ldots)$, we then define a new signed multi-partition $(\ula,\uep)$, where
\begin{align}
\label{def:ula_uep}
    \ula =( (\la_1),(\la_2),\ldots),
    \qquad
    \uep =(\ep_1^{\ell(\la^{(1)})},\ldots,\ep_r^{\ell(\la^{(r)})}).
\end{align}
We define
\begin{align}
   \label{eq:Sbfla}
S^{\bflaep} (\VV) = S^{\la^{(1)}} (\VV^{\ep_1}) \otimes \ldots \otimes S^{\la^{(r)}} (\VV^{\ep_r}),
\quad
S^{\bflaep} (\VV\!_\cZ) = S^{\la^{(1)}} (\VV^{\ep_1}\!\!\!\!\!_\cZ) \otimes \ldots \otimes S^{\la^{(r)}} (\VV^{\ep_r}\!\!\!\!\!_\cZ).
\end{align}
We now assume that $\bflaep$ is of $(n|m)$-type in the sense that $|\la\!^+|=n$ and $|\la\!^-|=m$.
Associated with $\bflaep$, we can form an $(n|m)$-sign sequence
\begin{align}  \label{eq:s_ep}
\bfs_{\bflaep}=(\ep_1^{|\la^{(1)}|}, \ldots, \ep_r^{|\la^{(r)}|}).
\end{align}

\begin{ex}
\label{ex:signedmulti-partition}
Consider a signed multi-partition $\bflaep$ with $\bfla =(\la^{(1)},\la^{(2)},\la^{(3)},\la^{(4)})$ and $\bfep=(+-+-)$, where
$\la^{(1)}= (3,3,1), \la^{(2)}=(4,2),\la^{(3)}=(2),$ and $\la^{(4)}=(3,1)$. Then $\ula=\big((3),(3),(1), (4),(2), (2), (3),(1)\big)$, $\uep=(+++--+--)$, and $\bfs_{\bflaep}=(+^7-^6+^2-^4)$. Compare Example \ref{ex:pyramid}(2).
\end{ex}

Given a partition $\la=(\la_1,\ldots, \la_\ell)$, we let $\texttt{Sym}_{\la} =\texttt{Sym}_{\la_\ell}\ldots\texttt{Sym}_{\la_1}$  denote the $q$-symmetrizer in $\mc H_{\la}:  =\mc H_{\la_\ell}\otimes \cdots \otimes \mc H_{\la_1}$ (a subalgebra of $\mc H_{|\la|})$.
For a multi-partition $\bfla =(\la^{(1)}, \ldots, \la^{(r)})$, we let $\texttt{Sym}_{\bfla} =\texttt{Sym}_{\la^{(1)}}\ldots\texttt{Sym}_{\la^{(r)}}$ denote the $q$-symmetrizer in $\mc H_{\bfla}:  =\mc H_{\la^{(1)}}\otimes \cdots \otimes \mc H_{\la^{(r)}}$ (a natural subalgebra of $\mc H_{n+m}$).

Recall the row reading $\rho(A)$ of a tableau $A$ in Section \ref{ssec:tab}. Let $\bflaep$ be of $(n|m)$-type, where $\bfla =(\la^{(1)}, \ldots, \la^{(r)})$ and $\bfep=(\ep_1,\ldots,\ep_r)$. Given $\bfA=(A^{(1)},\ldots,A^{(r)})\in \Row(\bflaep)$, the row reading $\rho(\bfA) \in \Z^{n+m}$ is understood to be the row reading $\rho(A^{(1)})$, followed by $\rho(A^{(2)})$, and so on.
We define the following subset of $\Z^{n+m}$:
\[
\Z^{\bflaep}_-:=\{\rho(\bfA) \in\Z^{n+m} \mid \bfA \in\Row(\bflaep)\},
\]
and so we have a bijection
\[
\Row(\bflaep) \stackrel{1{-}1}{\longrightarrow} \Z^{\bflaep}_-, \qquad
\bfA \mapsto \rho(\bfA),
\]
with the inverse map denoted by $f \mapsto \bfA_f.$ We shall refer to $f \in \Z^{\bflaep}_-$ as $\bflaep$-anti-dominant. This generalizes $\Z^{k,\ep}_-$ from \eqref{eq:Zk-}.

Extending the canonical map $\bbT^{|\la|} (\VV^\ep) \rightarrow S^{\la} (\VV^\ep)$ for a partition $\la$, we have a canonical map
\begin{align}  \label{eq:piS}
\pi: \bbT^{\bfs_{\bflaep}} \longrightarrow S^{\bflaep} (\VV).
\end{align}
Recall from \eqref{eq:Mf} the monomial basis $\{M_f\mid f\in \Z^{n+m}\}$ for $\bbT^{\bfs_{\bflaep}}$. Denote
\begin{align} \label{def:NA}
\Pi_\bfA =\pi(M_{\rho(\bfA)}), \quad \text{ for }\bfA\in \Row(\bflaep).
\end{align}
The following lemma follows immediately from the case when $\bfla$ is a partition.
\begin{lem}
\label{lem:stdbasis}
    We have a $\U$-module isomorphism $S^{\bflaep} (\VV) \cong \bbT^{\bfs_{\bflaep}} \texttt{Sym}_{\bfla}$ and a $\U\!_\cZ$-module isomorphism $S^{\bflaep}_\cZ (\VV)\! \cong \bbT^{\bfs_{\bflaep}}_\cZ\texttt{Sym}_{\bfla}$. Moreover, the isomophism matches the monomial bases $\{\Pi_\bfA\mid \bfA\in\Row(\bflaep)\}$ in $S^{\bflaep}_\cZ (\VV)$ and $\{M_f\texttt{Sym}_{\bfla}|f\in\Z^{\bflaep}_-\}$ in $\bbT^{\bfs_{\bflaep}}_\cZ\texttt{Sym}_{\bfla}$.
\end{lem}
Denote by $\widehat S^{\bflaep}(\VV)$ the completion of $S^{\bflaep}(\VV)$ with respect to the monomial basis in the partial order $\prec_{\bfs_{\bflaep}}$. Recall the dual canonical basis $\{\pounds_f\mid f\in \Z^{n+m}\}$ in the completion $\widehat{\bbT}^\bfs$ as in \eqref{dual:can:T}. By the isomorphism in Lemma \ref{lem:stdbasis}, we can view $S^{\bflaep}(\VV)$ (respectively, a suitable completion $\widehat S^{\bflaep}(\VV)$ with respect to the monomial basis in the Bruhat order) as a subspace of $\bbT^{\bfs_{\bflaep}}$ (respectively, $\widehat{\bbT}^{\bfs_{\bflaep}}$).

\begin{thm} \label{thm:DCB:Slambda}
    There exists a dual canonical basis $\{L_\bfA\mid \bfA\in \Row(\bflaep) \}$ on $\widehat S^{\bflaep}(\VV)$, which is characterized by $\overline{L_\bfA}=L_\bfA$ and
    \[
    L_\bfA \in \Pi_\bfA +\sum_{\bfA'\prec_{\bfs_{\bflaep}} \bfA} q^{-1}\Z[q^{-1}] \Pi_{\bfA'}.
    \]
    (Infinite sums are allowed here.) Moreover, we have the identification $L_\bfA=\pounds_{\rho(\bfA)}\texttt{Sym}_{\bfla}$.
\end{thm}

\begin{proof}
This is a natural extension of \cite[Theorem 10]{CCM23}, which dealt with the case when $\bfep$ is a standard sign sequence $(+^a, -^b)$.

By a standard argument (see an analogous proof of Theorem \ref{thm:DCB_P} for detail), we have
\begin{align*}
    \ov{\Pi_\bfA} \in \Pi_\bfA +  \sum_{\bfA'\prec_{\bfs_{\bflaep}} \bfA} \Z[q,q^{-1}] \Pi_{\bfA'},
\end{align*}
which usually involve infinite sums.
Hence, applying \cite[Lemma~ 24.2.1]{Lus10}, we obtain a dual canonical basis $\{L_\bfA\mid \bfA\in \Row(\bflaep)\}$
in $\widehat S^{\bflaep}(\VV\!_\cZ)$ with the desired properties.

Using the identification $\Pi_\bfA =M_{\rho(\bfA)}\texttt{Sym}_{\bfla}$ from Lemma \ref{lem:stdbasis}, we see that $\pounds_{\rho(\bfA)}\texttt{Sym}_{\bfla}$ satisfies the same characterization of the dual canonical basis element $L_\bfA$, and whence $L_\bfA=\pounds_{\rho(\bfA)}\texttt{Sym}_{\bfla}$.
\end{proof}

\subsection{$q$-exterior tensors}

We define the $k$-th $q$-exterior power, for $\ep=\pm$, to be
\[
\wedge^k(\VV^\ep) :=\bbT^k(\VV^\ep) \texttt{Ant}_k.
\]
Given a partition $\mu =(\mu_1,\ldots,\mu_\ell)$ and a sign $\ep$, we define the $\U$-modules
\begin{align*}
   \wedge^{\mu}(\VV^\ep) &:=\wedge^{\mu_1}(\VV^\ep) \otimes \cdots\otimes\wedge^{\mu_\ell}(\VV^\ep).
\end{align*}
Given a multi-partition $\bfmu =(\mu^{(1)}, \ldots, \mu^{(s)})$ and an associated sign sequence $\bfep =(\ep_1, \ldots, \ep_s)$, we define the $\U$-module
\begin{align} \label{wedgenueta}
\wedge^{\bfmue}(\VV) &:=\wedge^{\mu^{(1)}}(\VV^{\ep_1})\otimes \cdots \otimes\wedge^{\mu^{(s)}}(\VV^{\ep_s}).
\end{align}

Given a partition $\mu=(\mu_1,\ldots, \mu_\ell)$, we let $\texttt{Ant}_{\mu} =\texttt{Ant}_{\mu_1}\ldots\texttt{Ant}_{\mu_\ell}$  denote the $q$-anti-symmetrizer in $\mc H_{\mu}:  =\mc H_{\mu_1}\otimes \cdots \otimes \mc H_{\mu_\ell}$ (a subalgebra of $\mc H_{|\mu|})$.
For a multi-partition $\bfmu =(\mu^{(1)}, \ldots, \mu^{(s)})$, we let $\texttt{Ant}_{\bfla} =\texttt{Ant}_{\mu^{(1)}}\ldots\texttt{Ant}_{\mu^{(s)}}$ denote the $q$-anti-symmetrizer in $\mc H_{\bfmu}:  =\mc H_{\mu^{(1)}}\otimes \cdots \otimes \mc H_{\mu^{(s)}}$ (a natural subalgebra of $\mc H_{n+m}$). Let $w_{\bfmu}$ be the longest element in $\mf S_{\bfmu} =\mf S_{\mu^{(1)}}\times \ldots \times \mf S_{\mu^{(s)}}$, a Young subgroup of $\mf S_n\times \mf S_m$.

Recall the column reading $c(A)$ of a tableau $A$ in Section \ref{ssec:tab}. Let $\bfmue$ be of $(n|m)$-type, where $\bfmu =(\mu^{(1)}, \ldots, \mu^{(s)})$ and $\bfep=(\ep_1,\ldots,\ep_s)$. Given $\bfA=(A^{(1)},\ldots,A^{(s)})\in \Col(\bfmue)$, the column reading $c(\bfA) \in \Z^{n+m}$ is understood to be the column reading $c(A^{(1)})$, followed by $c(A^{(2)})$, and so on.
We define the following subset of $\Z^{n+m}$:
\[
\Z^{\bfmue}_+:=\{c(\bfA) \in\Z^{n+m} \mid \bfA \in\Col(\bfmue)\},
\]
and so we have a bijection $\Col(\bfmue) \rightarrow \Z^{\bfmue}_+, \bfA \mapsto c(\bfA)$, with the inverse map denoted by $f \mapsto \bfA_f.$ We shall refer to $f \in \Z^{\bfmue}_+$ as $\bfmue$-dominant. This generalizes $\Z^{k,\ep}_+$ from \eqref{eq:Zk+}.

For $f\in\Z^{\bfmue}_+$, we set $\mc K_f:=M_{fw_{\bfmu}}\texttt{Ant}_{\bfmu}$, and $\{\mc K_f\mid f\in\Z^{\bfmue}_+\}$ forms a monomial basis for $\wedge^{\bfmu}(\VV)$; we shall also write $\mc K_{\bfA_f}=\mc K_f$.

\subsection{The $\U$-modules $P^{\bflaep} (\VV)$}

Denote by $P^\la(\VV)$ the polynomial representation of $\U$ labeled by a partition $\la$. This is a straightforward $\infty$-rank version of the polynomial representations of $U_q(\mf{gl}_N)$ (for $N\ge n$) of highest weight $\la$, though at this infinite-rank limit $P^\la(\VV)$ is no longer of highest weight. Similarly, we have a $\U$-module $P^\la(\VV^-)$.

Denote by $\la'$ the transpose partition associated to $\la$.
For $\ep =\pm$, we consider the composition map
\begin{align} \label{def:xilambda}
\xi_\la:\wedge^{\la'} (\VV^\ep) \hookrightarrow \bbT^n(\VV^\ep) \stackrel{\calR_\la}{\longrightarrow} \bbT^n(\VV^\ep) \stackrel{\pi}{\twoheadrightarrow} S^{\la} (\VV^\ep).
\end{align}
Here, $\calR_\la$ denotes a suitable $R$-matrix and $\pi$ is the canonical projection; this is a reformulation of \cite[Page 39, (4.10)]{BK08}. It is well known (cf., e.g., \cite{Bru06}) that $\text{Im}\xi_\la$ is the irreducible $\U$-module $P^\la(\VV^\ep)$, and thus we have obtained the following maps:
\begin{align}  \label{EPS:V}
    \wedge^{\la'} (\VV^\ep) \twoheadrightarrow P^\la(\VV^\ep) \hookrightarrow S^{\la} (\VV^\ep).
\end{align}

\begin{ex}
    We have $P^{(n)}(\VV^\ep)= S^n (\VV^\ep)$ and $P^{(1^n)}(\VV^\ep)=\wedge^n (\VV^\ep)$.
\end{ex}

The $\U$-module $P^{\la}(\VV^\ep)$ admits the standard basis
\begin{align} \label{def:VA}
\{V_A:=\xi_\la(\mc K_A)\mid A \in \Std(\la,\ep)\}.
\end{align}
By \cite[Theorem 26]{Bru06} and \cite[(4.12)]{BK08} (which only consider the case $\ep=+$, but the extension to $\ep=-$ is trivial), we have that $\xi_\la(\mc L_A)\not=0$ if and only if $A\in\Std(\la,\ep)$, and furthermore the set $\{\xi_\la(\mc L_A)\mid A \in \Std(\la,\ep)\}$ forms the dual canonical basis in $P^\la(\VV^\ep)$ with respect to the standard basis $\{V_A\mid A\in\Std(\la,\ep)\}$.

We extend the above constructions to the setting of multi-partitions. Given  $\bfla =(\la^{(1)},\ldots, \la^{(r)})$ together with a sign sequence $\bfep=(\ep_1,\ldots,\ep_r)$, we define the $\U$-module and the $\U_\cZ$-module respectively:
\begin{align}
 \label{{eq:Plaep}}
P^{\bflaep} (\VV)
= \bigotimes_{i=1}^r P^{\la^{(i)}} (\VV^{\ep_i}),
\qquad
P^{\bflaep} (\VV\!\!_\cZ)
= \bigotimes_{i=1}^r P^{\la^{(i)}} (\VV^{\ep_i}\!\!\!\!\!_\cZ).
\end{align}
For $\bfA=(A^{(1)},\ldots, A^{(r)}) \in \Std(\bflaep)$, where $A^{(i)} \in \Std(\la^{(i)},\ep_i)$, we define the following element in $P^{\bflaep}_\cZ (\VV)$:
\begin{align} \label{def:DeltaA}
\Delta_\bfA :=\otimes_{i=1}^r V_{A^{(i)}}.
\end{align}
Then $\{\Delta_\bfA\mid \bfA\in \Std(\bflaep)\}$ forms a $\Q(q)$-basis for $P^{\bflaep} (\VV)$ and  an $\cZ$-basis for $P^{\bflaep} (\VV\!_\cZ)$. It will be called a {\em standard basis}.

\begin{ex}
\label{ex:special}
We list several distinguished special cases of $P^{\bflaep} (\VV)$ in \eqref{{eq:Plaep}} associated to standard sign sequences:
%, for partitions $\la=(\la_1,\ldots,\la_r)$ and $\mu=(\mu_1,\ldots,\mu_s)$:
\begin{enumerate}
    \item
    $S^{\la} (\VV) \otimes S^{\mu} (\VV^-)$, for $\bfla=\big((\la_1),\ldots,(\la_\ell),(\mu_1),\ldots,(\mu_p)\big)$, $\la=(\la_1,\ldots,\la_\ell), \mu=(\mu_1,\ldots,\mu_p)$, and $\bfep=(+^\ell,-^p)$.
    \item
    $P^{\la} (\VV) \otimes P^{\mu} (\VV^-)$, for $\bfla=(\la,\mu)$ and $\bfep=(+,-)$.
\item
$\bigwedge^{a_1}(\VV) \otimes \cdots \otimes \bigwedge^{a_r}(\VV)  \otimes \bigwedge^{b_1}(\VV^-) \otimes \cdots\otimes \bigwedge^{b_s}(\VV^-)$, for $\bfla =(\la^{(1)},\ldots,\la^{(r+s)})$ and $\bfep=(+^r,-^s)$, with $\la^{(i)} =(1^{a_i})$ and $\la^{(r+j)} =(1^{b_j})$, for $1\le i\le r$ and $1\le j\le s$.
\end{enumerate}
If we further set $s=0$ and $\mu=\emptyset$, then we obtain further specializations with the tensor factors involving $\VV^-$ dropped: $S^{\la} (\VV), P^{\la} (\VV)$, and $\bigwedge^{a_1}(\VV) \otimes \cdots \otimes \bigwedge^{a_r}(\VV)$.
\end{ex}

The embedding $\jmath_i: P^{\la^{(i)}}(\VV^{\ep_i}) \hookrightarrow S^{\la^{(i)}}(\VV^{\ep_i})$ from \eqref{EPS:V} leads to an embedding
\begin{align}
    \label{j:PS}
    \jmath: P^{\bflaep} (\VV) \longrightarrow S^{\bflaep} (\VV).
\end{align}

\subsection{Partial ordering on $\Std(\bflaep)$} \label{ssec:bruhat:order}

We define a partial order $\leq_{T,\ep}$ on standard tableaux as follows; cf.~\cite[\S2]{Bru06}. Fix $\ep=\pm$. Let $A, A' \in \Row(\la,\ep)$. Let $A_{\le r}$
denote the tableau obtained from $A$ by deleting all boxes in rows higher than the $r$th row, for $r\ge 1$; in particular, $A_{\le 1} =A$. Recall $\wt(A)$ from \eqref{wtP} and $\texttt{P}^+$ from \eqref{P+}. We declare that
\[
A'\le_{T,\ep} A \text{ if and only if } \ep\wt( A_{\leq r}) - \ep\wt(A'_{\leq r}) \in \texttt{P}^+ \, \forall r\ge 1 \text{ and } \wt( A) =\wt(A').
\]
As a subset of $\Row(\la,\ep)$, $\Std(\la,\ep)$ inherits the partial order $\leq_{T,\ep}$.
%{(Strictly speaking, $\Std(\la)\not\subseteq\Row(\la)$. I think $\text{Dom}(\la)\subseteq\Row(\la)$, where $\text{Dom}(\la)$ consists of those that have representatives in $\Col(\la)$. The partial order on $\text{Dom}(\la)$ is the same as the one defined on their representatives in $\Std(\la)\subseteq\Col(\la)$. Here, the partial on $\Col(\la)$ is defined by reading the columns instead of rows. Since the partial order on the exterior is different, this probably should be mentioned somewhere explicitly,.)}

We extend the partial order $\leq_{T,\ep}$ on tableaux to $\preceq_{T,\bfep}$ on tableaux associated to multi-partitions. Let $\bfA=(A^{(1)}, \ldots, A^{(r)}) \in \Std(\bflaep)$, where $A^{(i)}\in \Std(\la^{(i)},\ep_i)$. We define a sequence of (partial) weights for $\bfA$, for $1\le j\le r$:
\begin{align*}
\wt_{\bfep}^{j}(\bfA):= \sum_{i=j}^{r} \ep_i \wt(A^{(i)})
\in \texttt{P}.
\end{align*}
We set the (total) weight of $\bfA$ to be $\wt_{\bfep}(\bfA) :=\wt_{\bfep}^{1}(\bfA)$.

Generalizing the partial order $\preceq_\bfs$ on $\Z^{n+m}$, we define a partial order $\preceq_{T,\bfep}$ on $\Std(\bflaep)$ as follows. Let $\bfA=(A^{(1)}, \ldots, A^{(r)}), \bfA'=(A'^{(1)}, \ldots, A'^{(r)}) \in \Std(\bflaep).$ We say $\bfA'\preceq_{T,\bfep} \bfA$ if
$\wt_{\bfep}(\bfA') =\wt_{\bfep}(\bfA)$ and $\wt_{\bfep}^{i}(\bfA)- \wt_{\bfep}^{i}(\bfA') \in \texttt{P}^+ \; \forall i$; and moreover, $A'^{(i)} \le_{T,\ep_i}\!\! A^{(i)} \;\forall i$ in case $\wt_{\bfep}^{i}(\bfA)= \wt_{\bfep}^{i}(\bfA')\; \forall i$  (or equivalently, $\wt(A^{(i)})=\wt(A'^{(i)}) \;\forall i$). We write $\bfA'\prec_{T,\bfep} \bfA$ if $\bfA'\preceq_{T,\bfep} \bfA$ and $\bfA'\neq \bfA$.

\subsection{Dual canonical basis for $\widehat P^{\bflaep} (\VV)$}

Denote by $\widehat P^{\bflaep} (\VV)$ and $\widehat P^{\bflaep} (\VV\!_\cZ)$ the completion with respect to the standard basis in partial order $\prec_{T,\bfep}$. Lusztig \cite{Lus92} constructed (dual) canonical bases on tensor product modules of finite-dimensional highest weight modules more generally and somewhat differently. We have the following multi-partition generalization of \cite[Theorem 26]{Bru06}.

\begin{thm} \label{thm:DCB_P}
There exists a dual canonical basis $\{[L_\bfA \mid \bfA\in \Std(\bflaep)\}$ for $\widehat P^{\bflaep} (\VV)$, which is characterized by the bar invariance $\ov{L_\bfA}=L_\bfA$ and the property
    \[
    L_\bfA \in \Delta_\bfA + \sum_{\bfA'\prec_{T,\bfep} \bfA} q^{-1} \Z[q^{-1}] \Delta_{\bfA'}.
    \]
    (Infinite sums are allowed here.) Moreover, $\{[L_\bfA \mid \bfA\in \Std(\bflaep)\}$ forms a $\Z[q,q^{-1}]$-basis for $\widehat P^{\bflaep} (\VV\!_\cZ)$.
\end{thm}

\begin{proof}
Denote by $\psi$ the bar involution on the tensor product module $S^{\bflaep}(\VV)$ and on its submodule $P^{\bflaep} (\VV)$. We shall show that, for $\bfA\in \Std(\bflaep)$,
\begin{align} \label{bar:DeltaA}
    \psi(\Delta_{\bfA})\in \Delta_{\bfA} +\sum_{\bfA''\prec_{T,\bfep} \bfA}\Z[q,q^{-1}] \Delta_{\bfA''},
\end{align}
which may involve infinite sums.
For the sake of notational simplicity, we shall only write down the details for the case when $\bfla=(\la^{(1)},\la^{(2)})$, and hence $\bfA =(A^{(1)},A^{(2)}), \bfep=(\ep_1,\ep_2)$, and $\Delta_{\bfA} =V_{A^{(1)}} \otimes V_{A^{(2)}}$. The general case is easily obtained similarly by iterating $\psi$ and the argument below.

Thanks to the convention of the comultiplication $\Delta$ from \eqref{eq:Delta}, the quasi-$\mc R$-matrix for $\U$ takes the following form
\begin{align}\label{q-r-matrix}
    \Theta=1\otimes 1+\sum_{\gamma\in\texttt{P}^+\backslash 0}\Theta_{\gamma},
    \quad \text{for }\; \Theta_\gamma \in \U^+_{\gamma}\otimes \U^-_{-\gamma},
\end{align}
where $\U^\pm_{\pm\gamma}$ denotes the weight subspaces of halves of the quantum group $\U^\pm$ of weight $\pm\gamma$. Now, for the tensor product of two based modules $W'$ and $W$, a new bar involution $\psi$ on the tensor product module $W'\otimes W$ is defined by $\psi(v\otimes w) :=\Theta(\ov{v}\otimes\ov{w})$ by \cite[27.3.1]{Lus10}.

Recall \cite{Bru06} that
\[
\ov{V_{A^{(i)}}}\in V_{A^{(i)}}+\sum_{A'^{(i)}\prec_{T,\ep_i} A^{(i)}}\Z[q,q^{-1}]V_{A'^{(i)}},
\qquad \text{for } i=1,2.
\]
By definition of the partial order $\prec_{T,\bfep}$ and for weight reason, we have
\[
\Theta_\gamma(V_{A'^{(1)}}\otimes V_{A'^{(2)}}) \in \sum_{(A''^{(1)},A''^{(2)})\prec_{T,\bfep} (A'^{(1)},A'^{(2)})}\Z[q,q^{-1}] V_{A''^{(1)}}\otimes V_{A''^{(2)}}, \quad \text{ for }\gamma \in \texttt{P}^+\backslash 0.
\]
Hence by \eqref{q-r-matrix} we have
\begin{align*}
\psi(V_{A^{(1)}}\otimes V_{A^{(2)}})
&=\Theta(\ov{V_{A^{(1)}}}\otimes \ov{V_{A^{(2)}}})
=\ov{V_{A^{(1)}}}\otimes \ov{V_{A^{(2)}}} +\sum_{\gamma\in \texttt{P}^+\backslash 0} \Theta_\gamma (\ov{V_{A^{(1)}}}\otimes \ov{V_{A^{(2)}}})
\\
&\in V_{A^{(1)}}\otimes V_{A^{(2)}} +\sum_{(A''^{(1)},A''^{(2)})\prec_{T,\bfep} (A^{(1)},A^{(2)})}\Z[q,q^{-1}] V_{A''^{(1)}}\otimes V_{A''^{(2)}}.
\end{align*}
This proves \eqref{bar:DeltaA}. The theorem now follows by applying \cite[Lemma~ 24.2.1]{Lus10}.
\end{proof}

Here we use the same notation $L_\bfA$ for dual canonical bases in $P^{\bflaep} (\VV)$ as for $S^{\bflaep} (\VV)$ in Theorem \ref{thm:DCB:Slambda}. There is no ambiguity here thanks to the following.
\begin{prop}  \label{prop:sameDCB}
    The embedding $\jmath: \widehat P^{\bflaep} (\VV) \rightarrow \widehat S^{\bflaep} (\VV)$ (commutes with the bar involutions and) sends $L_{\bfA}$ in $\widehat P^{\bflaep} (\VV)$ to $L_{\bfA}$ in $\widehat S^{\bflaep} (\VV)$, for $\bfA\in \Std ({\bflaep})$.
\end{prop}

\begin{proof}
This is known when $\bfla=\la$ is a partition; see \cite{Bru06} (who considers the case for $\ep=+$, and the case for $\ep=-$ is the same), where it is shown that $V_A$, for $A\in \Std(\la)$, lies in the $\Z[q^{-1}]$-lattice spanned by $\{M_B \mid B\in \Row(\la)\}$. The general case here follows by the same argument by showing $\jmath (L_{\bfA})$ satisfies the characterization of a dual canonical basis element in $S^{\bflaep} (\VV)$. We skip the detail.
\end{proof}

%%%%%%%%
%%%%%%%%

\section{Character formulas of simple modules in super type $A$}
\label{sec:character}

In this section, we formulate general parabolic categories for $W$-superalgebras of type $A$, and introduce standard modules in these categories. We show that these categories and the irreducibles categorify the tensor product $\U$-modules and dual canonical bases given in the previous section.

\subsection{Gradings on $\glnm$ and Levi}

Let
\begin{align} \label{Inm}
I(n|m) :=\{1,\ldots,n, \overline{1}, \ldots, \overline{m}\},
\end{align}
which is totally ordered by $1<\ldots <n<\bar 1<\ldots <\bar m$. Let $\g=\glnm$ be the general linear Lie superalgebra, consisting of
$(n+m) \times (n+m)$-matrices over $\C$. It admits a $\Z_2$-grading $\g=\g_\oa \oplus \g_\ob$, where $\g_\oa =\gl_n \oplus \gl_m$. We can and often work with different matrix forms where the parity of rows (or columns) may not be standard.

We fix a non-degenerate invariant bilinear form on $\g$ denoted $(\cdot|\cdot)$. Let $e\in\g$ be a nilpotent element in $\g_\oa$ and $\{e,h,f\}$ be an $\mf{sl}(2)$-triple in $\g_\oa$ (and hence in $\g$). We let $\h$ be a Cartan subalgebra of $\g$ such that $h\in\h$.

For a subset $\mf a\subseteq\g$, we denote $\g_{\mf a}=\{x\in\g\mid[x,y]=0,\forall y\in\mf a\}$. Let
\begin{align}
    \label{eq:t}
\mf t:=\h_{e}=\{a\in\h\mid [a,e]=0\}.
\end{align}
Recall that a subalgebra $\mf r\subseteq\mf t$ is called a {\it full subalgebra} (cf. \cite[\S3.1]{BG13}) if the center of $\mf g_{\mf r}$ is equal to $\mf r$.

Let $T$ be the adjoint group of $\mf t$.
Let $\theta\in\mf t$ be an {\it integral element}, i.e., $\theta$ is an element in the cocharacter of $T$. The element $\theta\in\mf t$ determines a {\em minimal} full subalgebra $\mf r$ of $\mf t$ in the sense that $\theta\in\mf r\subseteq\mf t$ and $\theta$ is regular in $\mf r$; hence $\g_\theta =\g_{\mf r}$. We shall assume that
\[
\mf l:=\mf g_{\mf r} \text{ lies in } \g_\oa,
\]
which then must be reductive. The inclusion of subalgebras $0\subseteq\mf r\subseteq\mf t$ gives rise to an inclusion of Lie subalgebras $\g_{\mf t}\subseteq\g_{\mf r} \ (\subseteq \g_\oa )\subset \g$. It follows by definition of $\mf t$ in \eqref{eq:t} that $e \in \g_{\mf t}$, and hence we have
\begin{align}  \label{eq:el}
    e \in \mf l.
\end{align}
We have an $\ad\,\theta$-eigenspace decomposition of $\g$:
\begin{align}
\label{eq:eigen}
\begin{split}
    \g &=\bigoplus_{k\in\Z}\g_{\theta,k},
    \\
    \g_{\theta,k} &:=\{x\in\g\mid[\theta,x]=k x\},
    \quad \text{ with } \mf l=\g_{\theta,0}.
\end{split}
\end{align}
Let $\Phi$ be the root system for $(\g,\h)$, and $\g_\alpha$ be the root space for $\alpha \in \Phi$. Note that $\g_\alpha \subseteq \g_{\theta,\alpha(\theta)}.$

We choose a triangular decomposition
\begin{align}
		\g=\mf n^- \oplus \h \oplus \mf n
  \label{eq:tri}
\end{align}
to be compatible with \eqref{eq:eigen} in the following sense: $\alpha (\theta) >0$ for $\alpha\in \Phi$ implies that $\alpha \in \Phi^+$. Here $\Phi^+$ denotes the set of positive roots corresponding to the simple system $\Pi$ for $\g$ such that $\mf n =\oplus_{\alpha \in \Phi^+} \g_\alpha$. Denote
\begin{align}
\label{eq:Levi}
\Phi_{\mf l} =\Phi \cap \{\alpha \in \Phi\mid \alpha(\theta)=0\},
\quad
\Phi_{\mf l}^+ =\Phi^+ \cap \Phi_{\mf l},
\quad
\Pi_{\mf l}=\Pi \cap \Phi_{\mf l}.
\end{align}
Then $\mf l$ is a Levi subalgebra of $\g$ with simple system $\Pi_{\mf l}$, and
\begin{align}
  \label{eq:udecomp}
	\mf g=\mf u^-\oplus \mf l \oplus \mf u,
 \qquad
 \mf l= \h\oplus \bigoplus_{\alpha \in \Phi_{\mf l}} \fg_\alpha,
\end{align}
where
\begin{align}
 \label{eq:paracomp}
\quad\fu:=\bigoplus_{\alpha(\theta)>0} \fg_\alpha,
\qquad
\fu^-:=\bigoplus_{\alpha(\theta)>0} \fg_{-\alpha}.
\end{align}
Note that $\mf u \subset \mf n$ and $\mf u^- \subset \mf n^-$.

\subsection{Finite $W$-superalgebras}
\label{ssec:finiteW}

 We first recall the general construction of finite $W$-superalgebras, for any basic Lie superalgebra $\g=\g_\oa \oplus \g_\ob$; cf., e.g., \cite{Los10, Wan11}, before specializing to $\g=\glnm$.

Let $e\in \g_\oa$, and pick a good grading of $\g$ associated to $e$:
\begin{align} \label{goodgrading}
\g=\bigoplus_{j\in\Z}\g(j),
\end{align}
i.e., $e\in \g(2)$ such that $\ad e: \g(i) \rightarrow\g(i+2)$ is injective for $i\le -1$ and surjective for $i\ge -1$.
The element $e$ defines a linear map
$\chi:\g\rightarrow \C$ via $\chi(x):=(e|x).$
This leads to an even super-skewsymmetric bilinear form
$ \omega_\chi:\g\times\g\rightarrow\C$ defined by
 $\omega_\chi(x,y)=\chi([x,y]).$
This form restricts to a non-degenerate super-symplectic bilinear form $\omega_\chi:\g(-1)\times\g(-1)\rightarrow\C$. Note that $\dim\g(-1)_\oa$ is even, and we shall assume that $\dim\g(-1)_\ob$ is even; this assumption is always valid for $\g=\glnm$.

One checks by \eqref{eq:t} that $\ad \,\mf t$ preserves the super-symplectic form $\omega_\chi|_{\g(-1)}$. Hence we can choose an $\ad\,\mf t$-invariant  Lagrangian subspace $\underline{\mf l \mf s}$ of $\g(-1)$ with respect to $\omega_\chi$ and define
$\mf m:=\underline{\mf l \mf s}\oplus\bigoplus_{j<-1}\g(j).$
Note that $\chi$ is a character of $\mf m$ which vanishes on $\mf m_\ob$. Set
$ \mf m_\chi:=\{x-\chi(x)\mid x\in\mf m \}$
and let $I_\chi$ be the left ideal of $U(\g)$ generated by $\mf m_\chi$. Also let $Q_\chi$ denote the left $U(\g)$-module $U(\g)/I_\chi$. The finite $W$-superalgebra associated to $e$ is defined to be the associative superalgebra
\begin{align*}
   U(\g,e) =\text{End}_\g(Q_\chi)^{\text{opp}}.
\end{align*}
Therefore, $Q_\chi$ is a $\big(U(\g), U(\g,e)\big)$-bimodule. It is well known that we can identify
\begin{align*}
    U(\g,e)&\equiv \left(U(\g)/I_\chi\right)^{\ad\,\mf m}\\
    &=\{u+I_\chi\in U(\g)/I_\chi\mid [x,u]\in I_\chi,\forall x\in\mf m\}.
\end{align*}
Since the Lagrangian subspace $\underline{\mf l \mf s}$ is $\ad\,\mf t$-invariant, we have that $\mf m$ is $\mf t$-invariant, and thus, $\chi([t,x])=0$, for all $t\in\mf t$ and $x\in\mf m$. Therefore, we conclude that
\begin{align}  \label{eq:tUge}
    \mf t\subseteq U(\g,e).
\end{align}
Suppose that the good grading of the $W$-algebra $U(\g,e)$ is even, i.e., $\g(j)=0$ unless $j\in 2\Z$. Then,
\[
\mf p:=\sum_{j\ge 0}\g(j)
\]
is a parabolic subalgebra of $\g$ with Levi subalgebra $\g(0)$. Recalling that $\mf m=\sum_{j<0}\g(j)$, we may identify $U(\g)/I_\chi$ with $U(\mf p)$, which in turn allows us to view $U(\g,e)$ as a subalgebra of $U(\mf p)$. Let the homomorphism $U(\mf p)\rightarrow U(\g(0))$ be induced by the natural projection $\mf p\twoheadrightarrow\g(0)$. The Miura transform
\begin{align} \label{eq:Miura}
\texttt{MT}: U(\g,e)\longrightarrow U\!\left(\g(0)\right)
\end{align}
is the restriction of this homomorphism to $U(\g,e)$. A $\g(0)$-module can be regarded as a $U(\g,e)$-module via pull-back by the Miura transform $\texttt{MT}$.

Now we specialize the above construction to $\g=\glnm$. Let $\bflaep$ be a signed multi-pyramid. We label the boxes in each of the signed pyramids according to the column reading of the $\la^{(k)}$s from top down as in Example \ref{ex:pyramid}. The corresponding even nilpotent element is
\begin{align*}
    e=\sum_{i,j}E_{ij},
\end{align*}
where the sum is over all $i,j$ belonging to the same pyramid $\la^{(k)}$ with $\text{row}(i)=\text{row}(j)$ and $\text{col}(i)=\text{col}(j)-1$. The associated good grading is given by:
\begin{align*}
    \deg E_{ab}=2(\text{col}(b)-\text{col}(a)),\quad a,b\in I(m|n).
\end{align*}
There exists an $h\in\h$ such that the eigenvalues of $\ad h$ coincide with the degrees. For example, we can take $h=\sum_{i=1}^{n+m} (\wp+1-2\text{col(i)})E_{ii}$, and then $\text{ad}h(E_{ab})= (\deg E_{ab})E_{ab}$, where $\wp_k$ and \[
\wp:=\max\{\wp_k\mid 1\le k\le r\}
\]
are the numbers of columns of $\la^{(k)}$ and $\bflaep$, respectively. Denote
\begin{align}  \label{eq:qj}
q^{(k)}_{j} = \text{length of the $j$th column of } \la^{(k)}, \quad \text{ and }\quad
q^{\ep}_{j} =\sum_{k:\ep_k=\ep}q^{(k)}_{j},
\end{align}
for $\ep=\pm.$ Then we have $\g(0)=\oplus_{k=1}^r\oplus_{j=1}^{\wp_k} \gl_{q^+_{j}|q^-_{j}}$.

\subsection{An example}
\begin{ex} \label{ex:pyramid}
Consider a multi-pyramid with a multi-partition $\bfla =(\la^{(1)},\la^{(2)},\la^{(3)},\la^{(4)})$, where
$\la^{(1)}= (3,3,1), \la^{(2)}=(4,2),\la^{(3)}=(2),\la^{(4)}=(3,1)$, as shown in the leftmost figure in \eqref{YDiagram}. We can consider different sign sequences for $\bfla$.

\iffalse
\begin{align}\label{YDiagram}
\begin{ytableau}
    *(orange) \none & \none & \none&\none\\
    *(orange)\none & *(orange)\none  & *(orange) \none&\none \\
    *(orange) & *(orange)  & *(orange) &\none \\
              &   &\none  &\none\\
              &   & \script  & \script\\
       *(brown) & *(brown)  & \none&\none \\
        *(green)\\
        *(green) & *(green)  & *(green)  &\none
\end{ytableau}
\qquad\qquad
\begin{ytableau}
    *(orange)1 & \none & \none&\none\\
    *(orange)2 & *(orange)4  & *(orange) 6&\none \\
    *(orange)3 & *(orange)5  & *(orange) 7&\none \\
             8 & 10  &\none  &\none\\
             9 & 11  & \script 12 & \script 13
\\
       *(brown)\overline{1} & *(brown)\overline{2}  & \none&\none \\
        *(green)\overline{3}\\
        *(green)\overline{4} & *(green)\overline{5}  & *(green)\overline{6}  &\none
\end{ytableau}
\qquad\qquad
\begin{ytableau}
    *(orange)1 & \none & \none&\none\\
    *(orange)2 & *(orange)4 & *(orange) 6&\none \\
    *(orange)3 & *(orange)5 & *(orange) 7&\none \\
    \overline{1}  & \overline{3} &\none  &\none\\
    \overline{2} & \overline{4} & \overline{5}  & \overline{6}
\\
       *(brown)8 & *(brown)9  & \none&\none \\
        *(green) \overline{7}\\
        *(green)\overline{8} & *(green)\overline{9}  & *(green)\overline{10}  &\none
\end{ytableau}
\end{align}
\fi

\begin{align}\label{YDiagram}
%1: The colored pyramid structure
\begin{ytableau}
    *(orange) & \none & \none & \none \\
    *(orange) & *(orange) & *(orange) & \none \\
    *(orange) & *(orange) & *(orange) & \none \\
    \quad & \quad & \none & \none \\
    \quad & \quad & \scriptstyle  & \scriptstyle  \\
    *(brown) & *(brown) & \none & \none \\
    *(green) \\
    *(green) & *(green) & *(green) & \none
\end{ytableau}
\qquad\qquad
%2: First indexing scheme
\begin{ytableau}
    *(orange)1 & \none & \none & \none \\
    *(orange)2 & *(orange)4 & *(orange)6 & \none \\
    *(orange)3 & *(orange)5 & *(orange)7 & \none \\
    8 & 10 & \none & \none \\
    9 & 11 & 12 & 13 \\
    *(brown)\bar{1} & *(brown)\bar{2} & \none & \none \\
    *(green)\bar{3} \\
    *(green)\bar{4} & *(green)\bar{5} & *(green)\bar{6} & \none
\end{ytableau}
\qquad\qquad
%3: Second indexing scheme
\begin{ytableau}
    *(orange)1 & \none & \none & \none \\
    *(orange)2 & *(orange)4 & *(orange)6 & \none \\
    *(orange)3 & *(orange)5 & *(orange)7 & \none \\
    \bar{1} & \bar{3} & \none & \none \\
    \bar{2} & \bar{4} & \bar{5} & \bar{6} \\
    *(brown)8 & *(brown)9 & \none & \none \\
    *(green)\bar{7} \\
    *(green)\bar{8} & *(green)\bar{9} & *(green)\bar{10} & \none
\end{ytableau}
\end{align}
\vspace{2mm}

(1) Consider a sign multi-partition $\bflaep$ with $\bfla$ above and $\bfep=(+,+,-,-)$. In this case, $\g =\gl_{13|6}$ with an even good grading on $\g$ determined by the multi-pyramid.

We consider the multi-tableaux in the middle of \eqref{YDiagram} labeled by $I(13|6)$ in a column-standard color-block-wise fashion.
Let $e$ denote
the matrix $\sum_{i,j}E_{ij}$ summing over all $i, j \in I(13|6)$ such that $\text{row}(i) = \text{row}(j)$ and
$\text{col}(i) = \text{col}(j)-1$. Then $e$ is nilpotent and has Jordan form $\bfla$.
The parabolic subalgebra $\mf p$ is spanned by $\{E_{ij}\mid \text{col}(j)\ge \text{col}(i)\}$, whose Levi subalgebra $\mf h$ is spanned by $\{E_{ij}\mid \text{col}(j)= \text{col}(i)\}$.

The grading is even, e.g., $\deg E_{4,7}=\deg E_{1,\overline{2}}=\deg E_{6,13}=2$, $\deg E_{1,3}= \deg E_{6,7}= 0$, $\deg E_{10,2}=\deg E_{10,\overline{1}}=-2$. In addition, we have $\g(0)=\gl_{5|3} \oplus\gl_{4|2} \oplus\gl_{3|1} \oplus\gl_1$.

The subalgebra $\mf t$ of $\g$ is $8$-dimensional and spanned by the 8 row idempotent matrices $\mc I_i$, for $i\in I(5|3) =\{1,2,3,4,5,\overline{1},\overline{2},\overline{3}\}$, e.g., $\mc I_1=E_{11},\, \mc I_3=E_{33}+E_{77}+E_{11,11},\, \mc I_{\overline{1}}=E_{\bar{1}\bar{1}}+E_{\bar{4}\bar{4}} $, and so on.
We have a $\mf t$-module isomorphism between $U(\g,e)$ and $U(\g_e)$.
An integral element $\theta\in\mf t$ is of the form
\begin{align}  \label{theta}
   \theta=\sum_{i\in I(5|3)} \theta_i\mc I_i,
   \qquad \text{ for } \theta_i\in\Z.
\end{align}
We impose the following conditions on $\theta$:
\begin{align} \label{theta:standard}
\begin{split}
\theta_i \geq \theta_{i'} &
\text{ whenever row $i$ is higher than row $i'$,}
\\
\text{ and } \theta_i=\theta_{i'} & \text{ if and only if row $i$ and row $i'$ have the same color}.
\end{split}
\end{align}
In our example, it means that
$\theta_1=\theta_2=\theta_3 > \theta_4=\theta_5 > \theta_{\overline{1}} >\theta_{\overline{2}}=\theta_{\overline{3}}$.
Then the Levi $\mf l=\g_\theta$ is given by $\mf l \cong \gl_{\orange{7}} \oplus \gl_{6} \oplus \gl_{\brown{2}} \oplus \gl_{\green{4}}$ with four summands corresponding to the four color-blocks. Viewed as an element in $\g$ or in $\mf l$, $e$ admits a Jordan form $\bfla$.

(2) We can consider other sign sequences associated to the same $\bfla$ above, say, $\bfep=(+,-,+,-).$ In this case, $\g =\gl_{9|10}$ with an even good grading on $\g$ determined by the above multi-pyramid. Then $\g(0) =\gl_{4|4} \oplus\gl_{3|3} \oplus\gl_{2|2} \oplus\gl_1$.
As we impose the same conditions \eqref{theta:standard} on $\theta$, then the Levi $\mf l$ is formally the same as in (1).

(3) In this example, we ignore the colors in \eqref{YDiagram}, and chose $\theta$ in \eqref{theta} satisfying \eqref{theta:standard} to be the ``principal" $\underline{\theta}$ such that $\theta_i$ are distinct for different row $i$. Then the corresponding Levi $\underline{\mf l} \cong \gl_1\times \gl_3\times \gl_3\times \gl_2\times \gl_4\times \gl_2\times \gl_1\times \gl_3$ and  $\underline{\mf l}$ is naturally a subalgebra of $\mf l$ (in (1) or (2) above). Note that $e$ is principal nilpotent in $\underline{\mf l}$.
\end{ex}

\subsection{Parabolic categories $\Pcat$}

By a straightforward super generalization of \cite[Theorem 3.8]{BGK08}, we have a (non-unique) $\mf t$-module isomorphism:
$ %\begin{align}\label{eq:uge}
    U(\g_e)\cong U(\g,e).
$ %\end{align}
Thus we have an $\ad\, \theta$-eigenspace decomposition
\begin{align}  U(\g,e)=U(\g,e)_0+\sum_{k\in\Z}U(\g,e)_k.
\end{align}
Set
\begin{align}
\begin{split}
U(\g,e)_{\ge 0} &:=\sum_{k\ge 0}U(\g,e)_k,
\qquad
U(\g,e)_{>0}:=\sum_{k> 0}U(\g,e)_k
\\
U(\g,e)_{\#} &:=U(\g,e)_{\ge 0}\bigcap U(\g,e)U(\g,e)_{>0}.
\end{split}
\end{align}
Then $U(\g,e)_{\#}$ is a two-sided ideal of $U(\g,e)_{\ge 0}$. We have by \cite[Theorem 4.1]{Los12} a canonical morphism
\begin{align}
  \label{eqn:iso:1}
    U(\g,e)_{\ge 0} \longrightarrow U(\g,e)_{\ge 0}/U(\g,e)_{\#}\cong U(\mf l,e).
\end{align}
Denote by $\widetilde{\mc O}(\theta)$ the category of finitely generated $U(\g,e)$-modules $M$ such that for any $x\in M$ there exists $n_x\in\Z$ with $U(\g,e)_k x=0$ for all $k\ge n_x$.

Now, given a $U(\mf l,e)$-module $V$, we can regard it as a $U(\g,e)_{\ge 0}$-module via the pullback \eqref{eqn:iso:1} and then apply induction to construct a $U(\g,e)$-module:
\begin{align}\label{eq::ind:w}
{\mc I}^{\g}(V)=U(\g,e)\otimes_{U(\g,e)_{\ge 0}}V.
\end{align}
This defines an exact functor ${\mc I}^{\g}$ from the category of finitely generated $U(\mf l,e)$-modules to $\widetilde{\mc O}(\theta)$.

On the other hand, for $M\in\widetilde{\mc O}(\theta)$ we define
\[
{\mc R}_{\mf l}(M) =\big\{x\in M\mid ux=0,\forall u\in U(\g,e)_{>0}\big\},
\]
which is naturally a $U(\mf l,e)$-module by \eqref{eqn:iso:1}. The functor ${\mc R}_{\mf l}$ is right adjoint to ${\mc I}^{\g}$.

Let $\mc O(\theta)$ be the subcategory of $\widetilde{\mc O}(\theta)$ of $U(\g,e)$-modules $M$ for which $\dim{\mc R}_{\mf l}(M)<\infty$ (optionally, we can impose the additional condition that $M$ is $\mf t$-semisimple, but it makes no  difference on the irreducible character formulas we will present).  The functor ${\mc I}^{\g}$ restricts to an exact functor from the category of finite-dimensional $U(\mf l,e)$-modules to $\mc O(\theta)$ with right adjoint being the restriction of the functor ${\mc R}_{\mf l}$ above:
\begin{align*}
U(\mf l,e)\mod \stackrel{\mc I^{\g}}{\longrightarrow}
\mc O(\theta),
\qquad
\mc O(\theta) \stackrel{\mc R_{\mf l}}{\longrightarrow} U(\mf l,e)\mod.
\end{align*}

Now if $V$ is a finite-dimensional irreducible $U(\mf l,e)$-module, then ${\mc I}^{\g}(V)$ has a composition series and a unique irreducible quotient $L^{\theta}(V)$, according to \cite[Corollary~ 3.6, Proposition 3.7]{Los12} and \cite{BG13}. In the special case when $e$ is principal nilpotent in $\mf l$, the constructions of the full (non-parabolic) category $\mathcal O$ was given earlier in \cite{BGK08}.

\vspace{2mm}
From now on we specialize to $\fg=\glnm$.

Given a multi-partition $\bfla =(\la^{(1)},\ldots, \la^{(r)})$ and an associated sign sequence $\bfep=(\ep_1,\ldots,\ep_r)$, we recall the compositions $\la\!^+, \la\!^-$ from \eqref{def:la_plusminus}. We assume that $\bflaep$ is of $(n|m)$-type in the sense that $|\la\!^+|=n$ and $ |\la\!^-|=m$. Set
$\Wg =U(\fg,e_{\bflaep})$, and choose an integral element $\theta$ whose corresponding Levi
is $\mf l = \prod_{i=1}^r \gl_{|\la^{(i)}|}$. We shall restrict ourselves to an integer weight variant of the  category $\mc O(\theta)$ of $\Wg$-modules, denoted by $\Pcat$; by integer weights here we mean that all the entries of relevant multi-tableaux under consideration are integers, not arbitrary complex numbers.

The finite-dimensional irreducibles in $\Wl$ are $L(\mf l,\bfA)$, for $\bfA \in \Std(\bflaep)$, and from now on we will write $L(\bfA) :=L^\theta(L(\mf l,\bfA))$ in $\Pcat$. Then $\{L(\bfA) \mid \bfA \in \Std(\bflaep)\}$ are all the irreducible objects (up to isomorphisms) in $\Pcat$.

\subsection{Standard modules}

Let $\bflaep$ be a signed multi-partition with $\bfla=(\la^{(1)}, \ldots, \la^{(r)})$ and $\bfep=(\ep_1,\ldots,\ep_r)$. We shall define several distinguished $\Wg$-modules associated with multi-tableaux $\bfA=(A^{(k)})_{1\le k\le r}\in\Std(\bflaep)$.

First, we interpret each $\bfA$ as a weight of the Cartan subalgebra. Denote by $a^{(k)}_{ij}$ the entry in the $i$th row and $j$th column of $\bfA$ which lies in $A^{(k)}$ -- the row/column labels here follow the convention in Section \ref{ssec:tab}. Denote by $n_k$ and $\ell_k$ the size and the length of the partition $\la^{(k)}$, respectively, for each $k$. We fix an enumeration of boxes in the multi-pyramid $\bfla$ by $I(n|m)$ following the examples in \eqref{YDiagram}, and let $\kappa(i,j,k)\in I(n|m)$ denote the label of the box occupied by the entry $a^{(k)}_{ij}$.
With the notation in place, we interpret $\bfA$ as the weight that transforms the matrix $E_{\kappa(i,j,k),\kappa(i,j,k)}$ by the scalar
\begin{align}\label{wt:aijplus}
%  \ep_ka_{ij}^{(k)}+i-1-\sum_{t<k}\ell_t +\sum_{t<k}\ep_t(n_t-1).
\ep_ka_{ij}^{(k)}+i-1 +\sum_{t<k}\big(\ep_t(n_t-1)-\ell_t\big).
\end{align}
For fixed $k$ and $j$, denote by $A_{j}^{(k)}$ the $j$th column of the tableau $A^{(k)}$, and we define $M(A_{j}^{(k)},\ep_k)$ and $V(A_{j}^{(k)},\ep_k)$ to be the Verma and the finite-dimensional irreducible module over $\gl_{q_j^{(k)}}$ of highest weight as in \eqref{wt:aijplus}, respectively. Then, following \cite[\S7.3]{BK08} we define a standard module over $U(\gl_{n_k},\la^{(k)})$ as the pullback via the Miura transform \eqref{eq:Miura} of the tensor product of $\gl(q_j^{(k)})$-modules $V(A_{j}^{(k)},\ep_k)$:
\begin{align} \label{eq:VA}
    V(A^{(k)},\ep_k) :=V(A_{1}^{(k)},\ep_k) \boxtimes \cdots \boxtimes V(A_{\wp_k}^{(k)},\ep_k).
\end{align}
Note that $V(A^{(k)},\ep_k)$ is finite dimensional.

We define a finite-dimensional $\Wl$-module
\begin{align} \label{eq:VA:Levi}
V(\mf l, \bfA):=\otimes_{i=1}^r V(A^{(i)},\ep_i)
\end{align}
and then a $\Wg$-module
\begin{align}  \label{module:DeltaA}
    \Delta(\bfA) :={\mc I}^{\g} \big(V(\mf l, \bfA)\big)
    =\Wg\otimes_{\Wg_{\ge 0}} V(\mf l,\bfA).
\end{align}
(The notation $\bfA$ here implicitly keeps track of $\bfep$ thanks to $\bfA\in\Std(\bflaep)$.)
Clearly $\Delta(\bfA)$ is in the category  $\Pcat$, and it is called a {\em standard module}.

\begin{ex} \label{ex:fullcat}
    Associated to $(\bfla,\bfep)$ we have a new signed multi-partition $(\ula,\uep)$ from \eqref{def:ula_uep}. The category $\calP(\ula,\uep)$  is the full category $\calO$, denoted by $\Ofull$, for the W-superalgebra $\Wg$ with $U(\mf l, \uep\,\ula)$ as its Cartan. (In the non-super setting, i.e., $\bfep=(+^r)$, $\Ofull$ was introduced in \cite{BGK08}.) Moreover, for $\bfA \in \Row(\bflaep)= \Std(\uep\,\ula)$, the standard module $\Delta(\bfA)$ coincides with the Verma module $M(\bfA)$, as the Cartan $U(\mf l, \uep\,\ula)$ is abelian and $V(\mf l,\bfA)$ in \eqref{eq:VA:Levi} is one-dimensional.
(In contrast, in a general parabolic category, the standard modules may not coincide with the parabolic Verma modules.)
\end{ex}

\begin{lem}
    \label{lem:irrep}
We have a surjective homomorphism $\Delta(\bfA) \rightarrow L(\bfA)$, for $\bfA\in \Std(\bflaep)$.
\end{lem}

\begin{proof}
    By \cite[Theorem 7.13]{BK08}, $V(\mf l, \bfA)$ is a highest weight module with $L(\mf l,A)$ as its simple head.
    Hence there is a surjective homomorphism from $\Delta(\bfA)$, see \eqref{module:DeltaA}, to ${\mc I}^{\g} \big(L(\mf l, \bfA)\big) =\Wg\otimes_{\Wg_{\ge 0}} L(\mf l,\bfA)$, and the latter has $L(\bfA)$ as its simple head. The lemma is proved.
\end{proof}

\subsection{Characters of standard modules}

Let $\bfA \in \Std(\bflaep)$. Recall $q_j^\ep$ from \eqref{eq:qj}. For each column $j$, we can further define the parabolic Verma $\gl_{{q_j^+|q_j^-}}$-module $N(\bfA_j)$, relative to the Levi subalgebra $\bigoplus_{k=1}^r\gl_{q^{(k)}_j}$, of highest weight given by $E^{(k)}_{i+j,i+j}$ acting as the scalar \eqref{wt:aijplus} on the highest weight vector. Denote by
\begin{align} \label{def:moduleNA}
N(\bfA) =N(\bfA_{1})\boxtimes\cdots\boxtimes N(\bfA_{\wp}),
\end{align}
the $\Wg$-module in $\Pcat$ which is the pullback via Miura transform \eqref{eq:Miura} of the parabolic Verma $\g(0)$-module $N(\bfA_{1})\otimes\cdots\otimes N(\bfA_{\wp})$.

\begin{rem}
For a signed partition $\bfep\bfla$ we can further interpret the entries of a tableaux $\bfA$ as a $\rho$-shifted weight of $\gl_{n|m}$ where $\rho$ is the Weyl vector of $\gl_{n|m}$ that is uniquely determined by $(\rho,\varepsilon_1)=0$ and $(\rho,\alpha)=(\alpha,\alpha)/2$, for every simple root of the corresponding Borel subalgebra of $\gl_{n|m}$ determined by the sign sequence $\underline{\ep}$. Here, the matrices in $\gl_{n|m}$ are labeled by $I(n|m)$ as in the multi-pyramid \eqref{YDiagram}. We let $\mc N(\bfA)$ be the parabolic $\gl_{n|m}$-Verma module of this $\rho$-shifted highest weight. We have an automorphism $\ov{\eta}:U(\mf p)\rightarrow U(\mf p)$ given by
\begin{align*}
E_{ij}\stackrel{\ov{\eta}}{\longrightarrow} E_{ij}+\ep_k\delta_{ij}\bigg(q^{(k)}_1 - \sum_{r=1}^{\text{col}(j)} q^{(k)}_{r}\bigg), \qquad \text{ for }i,j \in I(n|m).
\end{align*}
Here $i$ appears as an entry in the pyramid $\la^{(k)}$ whenever $i=j$. Note that this is the signed multi-partition generalization of $\ov{\eta}$ introduced in \cite[(3.23)]{BK08} in the one-partition case. Twisting the action of $U(\g,e)$ by the automorphism $\ov{\eta}$ above, we can show similarly as in the proof of \cite[Lemma 8.17]{BK08} that  the ``Whittaker functor'' maps $\mc N(\bfA)$ to $N(\bfA)$. Here, the Whittaker functor is understood to be analogously defined as in \cite[\S8.5]{BK08}, i.e., first taking the full dual of each $\mf t$-weight space, then applying the Whittaker invariant functor, and then taking the dual again.
\end{rem}

The main result of this subsection is the following identity in the Grothendieck group $[\Pcat]$. One can view this as a ``character formula" for $\Delta(\bfA)$ in terms of $N(\bfA)$.

\begin{thm}
\label{thm:charStandard}
  The following identity holds in the Grothendieck group $[\Pcat]$:
\[
[\Delta(\bfA)] =[N(\bfA)], \qquad \text{ for } \bfA\in\Std(\bflaep).
\]
\end{thm}

Let us consider for a moment the full category $\Ofull$ ($=\calP(\ula,\uep)$) from Example \ref{ex:fullcat}. In this special case, the character identity in Theorem \ref{thm:charStandard} is  established in \cite{LP26}, as a highly nontrivial super generalization of \cite[Theorem 6.2, Corollary 6.3]{BK08}; note that the non-super case corresponds to $\bfep=(+^r)$. \cite[Theorem 6.2]{BK08} describes the Gelfand-Tsetlin character of $\Delta(\bfA)$ via a maximal commuative subalgebra of a Yangian. Given $B \in \Row(\bflaep)$, we denote its columns from left to right by $B_{1},\ldots,B_{\wp}$. Peng \cite{Pen21} established a super generalization of works of Brundan-Kleshchev, relating shifted super Yangians to finite $W$-superalgebras of type $A$. We record here the identity of Lu-Peng.

\begin{thm} \cite{LP26}
\label{thm:LPcharVerma}
    The following identity holds in the Grothendieck group $[\Ofull]$:
    \[
    [M(B)] =\big[M(B_{1}) \boxtimes \ldots \boxtimes M(B_{\wp})\big],
    \quad
    \text{ for } B \in \Row(\bflaep).
    \]
\end{thm}
Theorem \ref{thm:charStandard} is a parabolic generalization of this result in \cite{LP26} and our proof below relies on it.

\begin{proof} [Proof of Theorem \ref{thm:charStandard}]

Given a column-standard tableau $A$ (associated with a partition viewed as a left-justified pyramid), we denote by $C_{A}$ its column stabilizer, which is a product of symmetric groups.
Given a multi-tableau $\bfA =(A^{(1)},\ldots,A^{(r)}) \in \Std(\bflaep)$ associated with a multi-partition $\bfla=(\la^{(1)},\ldots,\la^{(r)})$, we denote $C_{\bfA}=C_{A^{(1)}}\times\ldots\times C_{A^{(r)}}$. In particular, the column stabilizers $C_{A^{(i)}_{j}}$ make sense the $j$th columns $A^{(i)}_{j}$ of $A^{(i)}$, and $C_{\bfA_{j}} =C_{A^{(1)}_{j}}\times \ldots \times C_{A^{(r)}_{j}}$.

The identities below in the proof are understood to hold in the Grothendieck group $[\Ofull]$. Moreover, the classes of Verma modules $[M(B)]$ or $[M(\mf l,B)]$, for some multi-tableau $B$, are always understood to be $[M(B')]$ or $[M(\mf l,B')]$, for a unique $B'\in \Row(\bflaep)$ which can be obtained from $B$ by permuting entries within rows.

For each $1\le i\le r$, by recalling the definition of $V(A^{(i)})$ from \eqref{eq:VA} and applying the Weyl charcter formula for finite-dimensional irreducible modules over Lie algebras of type $A$, we have
\begin{align*}
    [V(A^{(i)})] &=\sum_{\sigma^{i}=(\sigma_{1},\ldots,\sigma_{l}) \in C_{A^{(i)}_{1}}\times\ldots\times C_{A^{(i)}_{\wp_i}}}
    (-1)^{\ell(\sigma_{1})+\ldots+\ell(\sigma_{\wp_i})} \Big[M(\sigma_{1} A^{(i)}_{1})\boxtimes\ldots\boxtimes M(\sigma_{l} A^{(i)}_{\wp_i})\Big]
     \\
     &=\sum_{\sigma^{i} \in C_{A^{(i)}}} (-1)^{\ell(\sigma^{i})}\big[M(\sigma^{i} A^{(i)})\big],
\end{align*}
where the last equality follows by applying Theorem \ref{thm:LPcharVerma} and $C_{A^{(i)}} =C_{A^{(i)}_{1}}\times\ldots\times C_{A^{(i)}_{\wp_i}}$.
This gives us by definition \eqref{eq:VA:Levi} that
\begin{align*}
    [V(\mf l, \bfA)] = \big[\otimes_{i=1}^r V(A^{(i)})\big]
    &=\sum_{\sigma=(\sigma^{1},\ldots,\sigma^{r}) \in C_{A^{(1)}}\times\ldots\times C_{A^{(r)}}} (-1)^{\ell(\sigma^{1})+\ldots+\ell(\sigma^{r})}\big[\otimes_{i=1}^rM(\sigma^{i} A^{(i)})\big]
    \\
    &=\sum_{\sigma \in C_{\bfA}} (-1)^{\ell(\sigma)}\big[M(\mf l,\sigma \bfA)\big].
\end{align*}
Therefore, by the above identity and the definition \eqref{module:DeltaA} of $\Delta(\bfA)$ we have
\begin{align}
\label{eq:DeltaMA}
    [\Delta(\bfA)] &= \big[{\mc I}^{\g} \big(V(\mf l, \bfA)\big)\big] =\sum_{\sigma \in C_{\bfA}} (-1)^{\ell(\sigma)}\big[{\mc I}^{\g}\big(M(\mf l,\sigma \bfA)\big)\big]
    \\
    &=\sum_{\sigma \in C_{\bfA}} (-1)^{\ell(\sigma)}\big[\big(M(\sigma \bfA)\big)\big],
    \notag
\end{align}
where the last equality uses the transitivity of induction functors; cf. \cite[Lemma 3.5]{BG13}.

On the other hand, for each $1\le j \le \wp$, the Weyl character formula expresses the character of a parabolic Verma in terms of Verma modules as
\begin{align*}
    \big[N(\bfA_{j})\big] =\sum_{\sigma_{j}\in C_{\bfA_{j}}} (-1)^{\ell(\sigma_{j})}\big[M(\sigma_{j} \bfA_{j})\big].
\end{align*}
Using this formula and recalling $N(\bfA) =N(\bfA_{1})\boxtimes\cdots\boxtimes N(\bfA_{\wp})$ from \eqref{def:moduleNA}, we have
\begin{align}
    [N(\bfA)] &=\sum_{(\sigma_{1},\ldots,\sigma_{l}) \in C_{\bfA_{1}}\times\ldots\times C_{\bfA_{\wp}}} (-1)^{\ell(\sigma_{1})+\ldots+\ell(\sigma_{\wp})} \Big[M(\sigma_{1} \bfA_{1})\boxtimes\ldots\boxtimes M(\sigma_{l} \bfA_{\wp})\Big]
    \label{eq:NMA} \\
    &=\sum_{\sigma \in C_{\bfA}}(-1)^{\ell(\sigma)} \big[M(\sigma \bfA)\big],
    \notag
\end{align}
where the last equality follows by applying Lu-Peng's identity from Theorem \ref{thm:LPcharVerma}.

Now the theorem follows by comparing the identities \eqref{eq:DeltaMA} and \eqref{eq:NMA}.
\end{proof}

\subsection{Simple character formulas in $\Pcat$}

Denoted by $[\Pcat]^\wedge$ the downward completion of the Grothendieck group with respect to the basis of standard modules in partial order $\prec_{T,\bfep}$.

\begin{lem} \label{lem:bases}
    The completed Grothendieck group $[\Pcat]^\wedge$ admits a basis consisting of the standard modules $\{[\Delta(\bfA)] \mid \bfA\in \Std(\bflaep)\}$.
\end{lem}

The main result of this paper is the following Kazhdan-Lusztig type character formula for parabolic categories of modules of the $W$-superalgebra $\Wg$. Denote
\[
\widehat P^{\bflaep} (\VV_\Z)=\Z\otimes_{\Q(q)}\widehat P^{\bflaep} (\VV_\cZ),
\]
where $\Q(q)$ acts on $\Z$ by $q\mapsto 1$, and similar notation applies later on.

\begin{thm}
 \label{thm:char}
 \qquad
 \begin{enumerate}
     \item
There is a $\Z$-linear isomorphism
$\Psi: [\Pcat]^\wedge \rightarrow \widehat P^{\bflaep} (\VV_\Z),$
    which sends the classes of standard modules $[\Delta(\bfA)]$ to the standard basis $\Delta_\bfA$, for $\bfA \in \Std(\bflaep)$.
    \item
    The map $\Psi$ sends the classes of simple modules to the dual canonical basis, i.e., $\Psi([L(\bfA)]) = L_\bfA$, for $\bfA \in \Std(\bflaep)$.
 \end{enumerate}
\end{thm}

Recall from \eqref{def:ula_uep} the new signed multi-partition $(\ula,\uep)$ associated with $(\bfla,\bfep)$. Recall from Example \ref{ex:fullcat} that $\Ofull$ denotes the full category $\mathcal O$ of $\Wg$-modules. The Verma (=standard) modules $\Delta(\bfA)$ in $\Ofull$ admits a simple head denoted by $L(\bfA)$, for $\bfA\in \Row(\bflaep)$. We first establish the following special case of Theorem \ref{thm:char}.

\begin{thm}
 \label{thm:Char_fullO}
Let $\bflaep$ be a signed multi-partition of type $(n|m)$. Then the $\Z$-linear isomorphism
$\Psi: [\Ofull]^\wedge \rightarrow \widehat S^{\bflaep} (\VV_\Z)$, defined by sending $[M(A)] \mapsto \Delta_A$, for all $A$, maps $[L(A)]$ to $L_A$ for all $A\in \Row(\bflaep)$.
\end{thm}

\begin{proof}
This theorem follows by combining $\bfep$-variants of two main results from \cite{CCM23} and \cite{CW25}; those two paper treated categories of modules with respect to the standard root systems associated with the standard sign sequence $\bfep_{\bfs\bft}$. The $\bfep$-variants can be established in an entirely analogous fashion assuming that the Levi subalgebra $\mf l_\zeta$ associated with nilcharacter $\zeta$ is an even Levi subalgebra with respect to the $\bfep$-simple root system, using the notations loc.~cit. This assumption holds in the setting of the present paper.

An equivalence of categories was established in \cite[Theorem 44]{CCM23} between the category $\Ofull$ of $\Wg$-modules and a suitable category ${\mc Wh}(\bflaep)$ of Whittaker $\g$-modules, where $\Delta(\bfA)$ (and resp. $L(\bfA)$) are in correspondence with the standard Whittaker modules (and resp. the irreducible quotients).

Now the theorem follows by an equivalence of categories between $\Ofull$ of $\Wg$-modules and ${\mc Wh}(\bflaep)$ which matches the standard modules and the irreducibles. Such an equivalence is due to Losev \cite{Los12} in the reductive Lie algebra setting, and the desired super generalization given in \cite[Theorem 6.11]{CW25} relies crucially on the work of \cite{SX20} which provides a super generalization of Losev's decomposition theorem.
\end{proof}

\begin{proof} [Proof of Theorem~\ref{thm:char}]

Part (1) follows from Lemma \ref{lem:bases}.

We consider the following diagram:
\begin{equation}
\label{diagA2}
\begin{tikzcd}
{[}\Pcat{]}^\wedge \ar[r,"\jmath_o"]\ar[d,"\Psi"]&{[}\Ofull{]}^\wedge \ar[d,"\Psi\!_o"]\\
\widehat P^{\bflaep} (\VV_\Z) \ar[r,"\jmath"]& \widehat S^{\bflaep} (\VV_\Z)
\end{tikzcd}
\end{equation}
Here $\jmath_o$ is the natural inclusion map on the Grothendieck group level as $\Pcat$ is a subcategory of $\Ofull$. The horizontal maps $\jmath_o, \jmath$ are embeddings of $\Z$-modules, and the vertical maps are $\Z$-linear isomorphisms (all these maps can be made to be $\U_\Z$-module homomorphisms, but we do not need this fact here). We add a subscript to $\Psi\!_o$ to distinguish it from the other $\Psi$.

We show the diagram \eqref{diagA2} is commutative.
Recall from \eqref{eq:DeltaMA} appearing in the proof of Theorem \ref{thm:charStandard} that \[
\jmath_o([\Delta(\bfA)]) =\sum_{\sigma \in C_{\bfA}} (-1)^{\ell(\sigma)}\big[\big(M(\sigma \bfA)\big)\big].
\]

By Theorem~\ref{thm:Char_fullO}, $\Psi\!_o([M(\sigma \bfA)]) =M_{\sigma \bfA}$, and hence
\begin{align} \label{eq:PsiI}
\Psi\!_o\circ \jmath_o([\Delta(\bfA)]) =\sum_{\sigma \in C_{\bfA}} (-1)^{\ell(\sigma)}M_{\sigma \bfA}.
\end{align}

Recall $\bfA=(A^{(1)},\ldots,A^{(r)})$. By definitions \eqref{def:xilambda}, \eqref{EPS:V}, \eqref{def:VA} and \eqref{def:DeltaA}, at the $q=1$ limit we have
\begin{align}  \label{eq:IDeltaA}
    \jmath(\Delta_\bfA) = \otimes_{i=1}^r\xi(\mc K_{A^{(i)}})
    =\otimes_{i=1}^r \Big(\sum_{\sigma^i\in C_{A^{(i)}}} (-1)^{\ell(\sigma^i)}M_{\sigma^iA^{(i)}} \Big)
    =\sum_{\sigma\in C_{\bfA}} (-1)^{\ell(\sigma)}M_{\sigma \bfA}.
\end{align}
In the last equality, we have identified $\sigma=(\sigma^1,\ldots,\sigma^r)$, $\sigma\bfA=(\sigma^1A^{(1)},\ldots,\sigma^rA^{(r)})$, and $M_{\sigma\bfA}=M_{\sigma^1A^{(1)}}\otimes\ldots\otimes M_{\sigma^rA^{(r)}}$.

Comparing \eqref{eq:PsiI}--\eqref{eq:IDeltaA} and recalling the definition that $\Psi ([\Delta(\bfA)]) =\Delta_{\bfA}$, we conclude that $\Psi\!_o\circ \jmath_o([\Delta(\bfA)]) =\jmath\circ\Psi([\Delta(\bfA)]) $ for all $\bfA\in \Std(\bflaep)$, i.e., the diagram \eqref{diagA2} is commutative.

On the other hand, by Proposition \ref{prop:sameDCB}, we have $\jmath (L_\bfA) =L_\bfA$. By definition we have $\jmath_o([L(\bfA)])= [L(\bfA)]$. Hence from the commutative diagram \eqref{diagA2} and $\Psi\!_o([L(\bfA)])= L_\bfA$ by Theorem \ref{thm:Char_fullO}, we conclude that $\Psi([L(\bfA)])= L_\bfA$.
The theorem is proved.
\end{proof}

\subsection{Finite-dimensional simple character formulas}

In case where $\bfla=(\la,\mu)$ and $\bfep=(+,-)$, we denote this (maximal) parabolic category $\Pcat$ to be $\calF(\la|\mu)$.

\begin{thm} \label{thm:fd_Char}
Let $\bfla=(\la,\mu)$ and $\bfep=(+,-)$, for partitions $\la$ and $\mu$ of $n$ and $m$.
\begin{enumerate}
    \item
    The category $\calF (\la|\mu)$ consists of all finite-dimensional $\Wgl$-modules of integer weights. A complete list of simple modules in $\calF(\la|\mu)$ consists of $L(A^{(1)},A^{(2)})$, for $(A^{(1)},A^{(2)}) \in \Std(\la,+)\times\Std(\mu,-)$.
    \item
    The $\Z$-module isomorphism $\Psi: [\calF(\la|\mu)]^\wedge \rightarrow
    P^\la(\VV_\Z) \widehat\otimes P^\mu(\VV^-_\Z)$, defined by mapping $\Delta(A^{(1)},A^{(2)}) \mapsto \Delta_{(A^{(1)},A^{(2)})}$, sends  $[L(A^{(1)},A^{(2)})]$ to $L_{(A^{(1)},A^{(2)})}$.
\end{enumerate}
\end{thm}

\begin{proof}
We prove (1). Under the assumptions in the theorem, the Levi $\mf l$ coincides with the even subalgebra $\g_\oa=\gl_n\oplus\gl_m$. Hence, by the PBW theorem for $U(\g,e)$ in \cite[Theorem 0.1]{ZS15}, it follows that all standard modules $\Delta(A^{(1)},A^{(2)})$ are finite dimensional and so are the simples $L(A^{(1)},A^{(2)})$, for $(A^{(1)},A^{(2)}) \in \Std(\la,+)\times\Std(\mu,-)$.

Parts (2) is a reformulation of Theorem \ref{thm:char}(2) in this special case.
\end{proof}

\begin{rem}
    The finite-dimensional irreducible $\Wgl$-modules (without imposing the integer weight condition) can be classified by modifying the proof of Theorem~\ref{thm:fd_Char}. Indeed, choosing again the Levi subalgebra $\mf l=\g_\oa$, we see that every finite-dimensional irreducible $U(\glnm,e)$-module is the head of a parabolic Verma module induced from a finite-dimensional $U(\g_\oa,e)$-module. Now, using the PBW theorem for $\Wgl$ again, it follows the finite-dimensional irreducibles are parametrized by the same parameter set as for finite-dimensional irreducible $U(\g_\oa,e)$-modules. The latter parameter set for finite $W$-algebra of type $A$ is known; cf. \cite{BK08}.
\end{rem}

\bibliographystyle{alpha}
\bibliography{Wchar.bib}

\end{document}